\documentclass{birkjour}
\usepackage{xcolor}
\usepackage{epsfig}
\usepackage[active]{srcltx} % SRC Specials: DVI [Inverse] Search
\usepackage{pdfsync}
\usepackage{manfnt}
\usepackage{hyperref}
\newtheorem{thm}{Theorem}[section]
 \newtheorem{cor}[thm]{Corollary}
 \newtheorem{lem}[thm]{Lemma}
 \newtheorem{prop}[thm]{Proposition}
 \theoremstyle{definition}
 \newtheorem{defn}[thm]{Definition}
 \theoremstyle{remark}
 \newtheorem{rem}[thm]{Remark}
 \newtheorem*{ex}{Example}
 \numberwithin{equation}{section}

\def\L2{L^{2}}

\def\A{\mathcal{A}}
\def\B{\mathcal{B}}

\def\E{\mathcal{E}}
\def\B{\mathcal{B}}
\def\D{\mathcal{D}}

\def\N{\Bbb{N}}
\def\R{\mathbb{R}}
\def\C{\Bbb{C}}

\def\m1{^{-1}}

\def\F{\mathcal{F}}

\def\L{\mathcal{L}}
\def\RR{\mathcal{R}}
\def\o{\omega}
\def\O{\Omega}

\begin{document}

\title[]{KMS Dirichlet forms, coercivity and superbounded Markovian semigroups on von Neumann algebras}%
%----------Author 1
\author{Fabio E.G. Cipriani}%
\address{Dipartimento di Matematica, Politecnico di Milano, piazza Leonardo da Vinci 32, 20133 Milano, Italy.}
\email{fabio.cipriani@polimi.it}
%----------Author 2
\author{Boguslaw Zegarlinski}%
\address{Institut de Mathematiques de Toulouse, CNRS UMR5219, 118, route de Narbonne F-31062 Toulouse Cedex 9,  France.}
\email{boguslaw.zegarlinski@math.univ-toulouse.fr}

\thanks{This work has been supported by Laboratorio Ypatia delle Scienze Matematiche I.N.D.A.M. Italy (LYSM) and G.N.A.M.P.A.-I.N.D.A.M. Italy 2023.}
\subjclass{%https://zbmath.org/static/msc2020.pdf
47C15, %Linear operators in $C^*$- or von Neumann algebras
46L57 %Derivations, dissipations and positive semigroups in C?-algebras
47D06, %One-parameter semigroups and linear evolution equations
47L90, %Applications of operator algebras to the sciences
7N60, %Applications of operator theory in chemistry and life sciences
46L51, %Noncommutative measure and integration
}
\date{March 10, 2023, first revisionon October 23, 2023, second revision 30 December 2023}
%\dedicatory{To Len Gross}
% ----------------------------------------------------------------

%-----------------------------------------------------------------
\begin{abstract}
We introduce a construction of Dirichlet forms on von Neumann algebras $M$ associated to any eigenvalue of the Araki modular Hamiltonian of a faithful normal non tracial state, providing also conditions by which the associated Markovian semigroups are GNS symmetric. The structure of these Dirichlet forms is described in terms of spatial derivations. Coercivity bounds are proved and the spectral growth is derived. We introduce a regularizing property of positivity preserving semigroups (superboundedness) stronger than hypercontractivity, in terms of the symmetric embedding of $M$ into its standard space $L^2(M)$ and the associated noncommutative $L^p(M)$ spaces. We prove superboundedness for a special class of positivity preserving  semigroups and that some of them are dominated by the Markovian semigroups associated to the Dirichlet forms introduced above, for type I factors $M$. These tools are applied to a general construction of the quantum Ornstein-Uhlembeck semigroups of the Canonical Commutation Relations CCR and some of their non-perturbative deformations.
\end{abstract}
\maketitle
%\tableofcontents
 ----------------------------------------------------------------------------------------------------------------
\section{Introduction and description of the results.}

The structure of completely Dirichlet forms with respect to lower semicontinuous, faithful traces on von Neumann algebras is well understood in terms of closable derivations taking values in Hilbert bimodules (see \cite{cipsau} and the recent \cite{ver}, \cite{Wirth}). However, for applications to Quantum Statistical Mechanics (see \cite{bkp1, bkp2, bra1, bra2, c3, c4, moz, mz1, m1, m3, p1, p2, p3, z}) and Quantum Probability (see \cite{cfl, m2})
or to deal with general Compact Quantum Groups, is unavoidable to consider quadratic forms which are Markovian with respect to non tracial states or weights. Concerning the structure of Dirichlet forms of GNS-symmetric Markovian semigroups, one is invited to consult the recent \cite{ver}, \cite{Wirth}. \\
In QSM, for example, the relevant states one wishes to consider are the KMS equilibria of time evolution automorphisms which are non tracial at finite temperature. In the CQGs situation, on the other hand, the Haar state is a trace only for the special subclass of CQGs of Kac type. In several most studied CQGs the Haar state is not a tracial state, as for examples for the special
unitary CQGs $SU_q(N)$. In this framework a detailed understanding has been found for the completely Dirichlet forms generating translation invariant completely Markovian semigroups of Levy quantum stochastic processes. The construction relies on the Sch\"urmann cocycle associated to the generating functional of the process (see \cite{cfk}).\\
On the other hand, a general construction of completely Dirichlet forms on the standard form of a $\sigma$-finite von Neumann algebra with respect
to a faithful, normal state in the sense of \cite{c1, c2, gl1, gl2, gl3}, has been introduced in \cite{sqv,mz1,mz2} and by Y.M. Park and his school (see \cite{p1,p2,p3, bkp1,bkp2}) with applications to QSM of bosons and fermions system and their quasi-free states. In this approach the Dirichlet forms depend upon the explicitly knowledge of the modular automorphisms group of the state.
\vskip0.1truecm
In this work we formulate a general and natural construction of a completely Dirichlet form, Markovian with respect to a fixed normal, faithful state $\o_0$, associated to each non zero and not necessarily discrete eigenvalue of the Araki modular Hamiltonian $\ln\Delta_0$. Hence, by superposition, one has a malleable tool to construct completely Dirichlet forms and completely
Markovian, modular symmetric, semigroups starting from the spectrum of the modular operator $\Delta_0$ or its associated Araki modular Hamiltonian $\ln\Delta_0$. Compared to Park's approach, this has the advantage to avoid the explicit use of the modular automorphism group. The present method generalizes the construction of bounded Dirichlet form of \cite{c1} Proposition 5.3 and that of unbounded Dirichlet forms of \cite{c1} Proposition 5.4, removing the assumption of self-adjointness and affiliation to the centralizer for the coefficients.\\
The framework of the construction is that of Dirichlet forms and Markovian semigroups on standard forms $(M, L^2(M), L^2_+(M), J)$ of von Neumann algebras $M$ as in \cite{c1} and related modular theory (\cite{ara1,ara2,bra1,tak,sk,sz}). In particular, we associate in Section 2, a one-parameter family of unbounded, $J$-real, non negative, densely defined, closed quadratic forms $(\E_Y^\lambda,\F_Y^\lambda)$ on $L^2(M)$ satisfying the first Beurling-Deny condition
to each densely defined, closed operator $(Y,D(Y))$ affiliated to $M$, thus generating $C_0$-continuous, contractive semigroups on $L^2(M)$ which are positivity preserving (in the sense that they leave globally invariant the positive self-polar cone $L^2_+(M)$). Moreover, the quadratic form
$(\E_Y^\lambda,\F_Y^\lambda)$ is Markovian with respect to the cyclic vector $\xi_0\in L^2(M)_+$ representing $\o_0$, in the strong sense that $\E_Y^\lambda[\xi_0]=0$, if and only if $\xi_0$ lies in the domain both of $Y$ and its adjoint $Y^*$ and $\xi:=Y\xi_0$ is an eigenvector of the modular operator $\Delta_0$ associated to the non zero eigenvalue $\lambda>0$. This construction applies, in particular, to any eigenvector $\xi$ of any non zero eigenvalue of $\Delta_0$.\\
Further, we investigate the fact that, by definitions, each $(\E_Y^\lambda,\F_Y^\lambda)$ is the quadratic form of an $M$-bimodule derivations $(d_Y^\lambda, D(d_Y^\lambda))$ on the standard bimodule $L^2(M)$. In particular we show that in the Markovian case both $(\E_Y^\lambda,\F_Y^\lambda)$ and $(d_Y^\lambda, D(d_Y^\lambda))$ are represented by the symmetric embedding on $L^2(M)$ of the unbounded, spatial derivations $\delta_Y:=i[Y,\cdot]$ on $M$ provided by the operator $(Y,D(Y))$ affiliated to $M$.\\
In the subsequent Section 3, we prove natural lower bounds for the Dirichlet form $(\E_Y^\lambda,\F_Y^\lambda)$ in terms of the quadratic forms of the affiliated operators $Y^*Y$, $YY^*$, $[Y,Y^*]$ and derive implications on the lower boundedness and discreteness of spectrum of $(\E_Y^\lambda,\F_Y^\lambda)$.\\
\noindent
By the general theory, using the symmetric embeddings of the von Neumann algebra $M$ into the standard Hilbert space $L^2(M)$ and the embedding of
$L^2(M)$ into the predual space $M_*=L^1(M)$, provided by the modular theory of the state $\o$, completely Markovian semigroups $T_t$ on $L^2(M)$ extend to completely (Markovian) contractive semigroups on $M$ and on $L^1(M)$ (weak$^*$-continuous in the former case and strongly continuous
in the latter one).\\
In Section 4, we introduce an extra regularity property of positivity preserving semigroups called {\it superboundedness} as the boundedness of $T_t$ from $L^2(M)$ to $M$ for all $t>t_0$ and some $t_0\ge 0$. In case
$t_0=0$ we call this property {\it ultraboundedness}. We prove that superboundedness holds true with respect to a finite temperature Gibbs state
$\o(\cdot)={\rm Tr\,}(\cdot e^{-\beta_0 H_0})/{\rm Tr }(e^{-\beta_0 H_0})$ on a type I$_{\infty}$ factor $M$, for the semigroup generated by the generalized sum $H_0\dot+JH_0J$ and that the property is stable with respect to domination of positivity preserving semigroups.\\ \noindent
In Section 5 we apply the framework above to investigate a class of Dirichlet forms associated on a type I$_{\infty}$ factor which are Markovian with respect to a Gibbs state of the Number Operator of a representation of the CCR algebra. The construction fully generalizes that of Quantum Ornstein-Uhlenbeck semigroups introduced in \cite{cfl}. In particular we prove the subexponential spectral growth rate of the generator and the domination of the Markovian semigroup with respect to the semigroup generated by $H_0\dot+JH_0J$ (this special class of semigroups is discussed in Appendix 7.1).\\
\noindent
In Section 6 we apply the tools developed in the previous sections to construct Dirichlet forms associated with dynamics generated by deformations of the Number Operator.\\
\noindent
In Appendix we represent the generators of a class of positivity preserving semigroups as generalized sums and we clarify {\it superboundedness} for abelian von Neumann algebras.

\section{Dirichlet forms and derivations on von Neumann algebras standard forms}
Let $(M,L^2(M),L^2_+(M),J)$ be a standard form of a $\sigma$-finite von Neumann algebra (for this subject and the related modular theory we refer to \cite{bra1,bra2,tak,sz}).\\
Let $\o_0$ be the faithful normal state on $M$ represented by the cyclic vector $\xi_0\in L^2_+(M)$ as
\[
\o_0(x)=(\xi_0|x\xi_0)_{L^2(M)}\qquad x\in M.
\]
The anti-linear, densely defined operator on $L^2(M)$ defined on the left Hilbert algebra by
\[
M\xi_0\ni x\xi_0\mapsto x^*\xi_0\qquad x\in M,
\]
is closable. Its closure $S_0$ has a polar decomposition $S_0=J\Delta_0^{1/2}$ where the anti-unitary part $J$ is called the modular conjugation
and $\Delta_0:=S_0^*S_0$ is a densely defined, self-adjoint, positive operator on $L^2(M)$, called the modular operator of $\o_0$, defining the modular automorphism group of $M$ by
$\sigma^{\o_0}_t(x):=\Delta_0^{it}x\Delta_0^{-it}$ for $x\in M$ and $t\in \R$. On the w$^*$-dense, involutive, sub-algebra of its analytic elements $M_0\subseteq M$, the modular group can be extended to any $t\in\mathbb{C}$. For any $x,y\in M_0$ and $z,w\in\C$, this extension satisfies
\[
\sigma^{\o_0}_z(xy)=\sigma^{\o_0}_z(x)\sigma^{\o_0}_z(y),\quad \sigma^{\o_0}_{z+w}(x)= \sigma^{\o_0}_z(\sigma^{\o_0}_w(x)),\quad (\sigma^{\o_0}_z(x))^*=\sigma^{\o_0}_{\bar z}(x^*).
\]
We will make use of the symmetric embedding of $M$ into its standard Hilbert space $L^2(M)$:
\[
i_0:M\to L^2(M)\qquad i_0(x):=\Delta_0^{1/4}x\xi_0.
\]
Among its properties we recall that it is weak$^*$-continuous, injective with dense range and positivity preserving in the sense that $i_0(x)\in
L^2(M)_+$ if and only if  $x\in M_+$. Also it maps the closed and convex set of all $x\in M_+$ such that $0\le x\le 1$ onto the closed and convex set of all $\xi\in L^2_+(M)$ such that $0\le\xi\le\xi_0$. The projection of a $J$-real vector $\xi=J\xi\in L^2(M)$ onto the closed, convex set $\xi_0-L^2_+(M)$ wil be denoted by $\xi\wedge\xi_0$.
\vskip0.2truecm\noindent
A {\it Dirichlet form} \cite{c1} Definition 4.8 with respect to $(M,\o_0)$ is a lower bounded and lower semicontinuous quadratic form
\[
\E:L^2(M)\to (-\infty,+\infty],
\]
with domain $\F:=\{\xi\in L^2(M): \E[\xi]<+\infty\}$, satisfying the properties
\vskip0.1truecm\noindent
i) $\F$ is dense in $L^2(M)$,\vskip0.2truecm\noindent
ii) $\E[J\xi]=\E[\xi]$ for all $\xi\in L^2(M)$ (reality),\vskip0.2truecm\noindent
iii) $\E[\xi\wedge\xi_0]\le\E[\xi]$ for all $\xi=J\xi\in L^2(M)$, (Markovianity).
\vskip0.2truecm\noindent
$(\E,\F)$ is said to be a {\it completely Dirichlet form} if its ampliation on the algebra\\
$(M\otimes M_n(\mathbb{C}),\o_0\otimes {\rm tr}_n)$ defined by
\[
\E^n:L^2(M\otimes M_n(\mathbb{C}),\o_0\otimes {\rm tr}_n)\to [0,+\infty]\qquad \E^n[[\xi_{i,j}]_{i,j=1}^n]:=\sum_{i,j=1}^n\E[\xi_{i,j}]
\]
is a Dirichlet form for all $n\ge1$ (${\rm tr}_n$ denotes the tracial state
on the matrix algebra $M_n(\mathbb{C})$).\\
\vskip0.1truecm
A $C_0$-continuous, self-adjoint semigroup $\{T_t:t\ge0\}$ on $L^2(M)$ is called
\vskip0.1truecm\noindent
i) {\it positivity preserving} if $T_t\xi\in L^2_+(M)$ for all $\xi\in L^2_+(M)$ and $t\ge 0$;
\vskip0.1truecm\noindent
ii) {\it Markovian with respect to $\o_0$} if it is positivity preserving
and for $\xi=J\xi\in L^2(M)$
\[
0\le \xi\le \xi_0\quad \implies \quad 0\le T_t \xi\le \xi_0\qquad t\ge 0;
\]
iii) {\it completely positive (resp. Markovian)} if the extensions $T^n_t:=T_t\otimes I_n$ to $L^2(M\otimes M_n(\mathbb{C}),\o_0\otimes {\rm tr}_n)$ are positivity preserving (resp. Markovian) semigroups for all $n\ge 1$.\\
In \cite{c1} Definition 2.8, property ii) above, Markovianity, was indicated as sub-Markovianity.
\vskip0.2truecm\noindent
As a result of the general theory, Dirichlet forms are automatically nonnegative and Markovian semigroups are automatically contractive see \cite{c1} Proposition 4.10 and Theorem 4.11.\\
Dirichlet forms $(\E,\F)$ are in one-to-one correspondence with Markovian
semigroups $\{T_t:t\ge 0\}$: the self-adjoint, positive operator $(H,D(H))$ associated to $(\E,\F)$ by $\E[\xi]=\|\sqrt{H}\xi\|^2_{L^2(M)}$ for all $\xi\in \F$, being the semigroup generator $T_t=e^{-tH}$, $t\ge 0$.\\
$C_0$-continuous, self-adjoint, positivity preserving semigroups are in one-to-one correspondence with nonnegative, densely defined, real, lower semicontinuous quadratic forms satisfying the following {\it first Beurling-Deny condition} (weaker than Markovianity)
\[
\xi=J\xi\in\F\quad \Rightarrow\quad \xi_\pm\in\F\qquad\text{and}\qquad \E(\xi_+|\xi_-)\le 0,
\]
equivalently stated (see \cite{c1} Proposition 4.5 and Theorem 4.7]) as
\[
\xi=J\xi\in\F\quad \Rightarrow\quad |\xi|\in\F\qquad\text{and}\qquad \E[|\xi|]\le \E[\xi],
\]
On the other hand, the first Beurling-Deny condition and the {\it conservativeness} condition
\[
\xi_0\in\F,\qquad \E[\xi_0]=0
\]
together imply the Markovianity of closed forms $(\E,\F)$ (see \cite{c1} Lemma 2.9 and Theorem 4.11).

\subsection{Dirichlet forms associated to eigenvalues of the modular operators}
The forthcoming construction of Dirichlet forms is based on the following well known fact (see \cite{bra1} Proposition 2.5.9, \cite{sk,tak} page 19; see also \cite{ara2} where von Neumann algebras with states having the logarithmic of the modular operators with spectrum consisting only of isolated eigenvalues are characterized).
\vskip0.2truecm\noindent
We recall that a densely defined, closed operator $(Y,D(Y))$ on $L^2(M)$ is affiliated to $M$ if for any $z'\in M'$ and any $\xi\in D(Y)$ one has $z'D(Y)\subseteq D(Y)$ and $Y(z'\xi)=z'(Y\xi)$ or, equivalently, if and only if its graph $\mathcal{G}(Y)\subset L^2(M)\oplus L^2(M)$ is left globally invariant $(z'\oplus z')\mathcal{G}(Y)\subseteq \mathcal{G}(Y)$ under the action of $z'\oplus z'\in M'\oplus M'$, for any $z'\in M'$ (see \cite{sz}).\\
For any operator $(Y,D(Y))$ affiliated to $M$, the operator $j(Y):=JYJ$ is affiliated to $M'$.
\begin{lem}
For any $\xi\in D(S_0)=D(\Delta_0^{1/2})$ there exists a densely defined, closed operator $(Y,D(Y))$ affiliated to $M$ such that
\vskip0.2truecm\noindent
i) $\xi_0\in D(Y)\cap D(Y^*)$,
\vskip0.2truecm\noindent
ii) $\xi=Y\xi_0$ and $S_0(\xi)=Y^*\xi_0$.
\vskip0.2truecm\noindent
iii) Among the operators $(Y,D(Y))$ with the properties i) and ii) above,
there exists a minimal one $({\overline{Y_0}}, D({\overline{Y_0}}))$ obtained as the closure of the closable operator $(Y_0,D(Y_0))$ defined by
\[
D(Y_0):=M'\xi_0,\qquad Y_0(y'\xi_0):=y'\xi.
\]
\end{lem}
\begin{proof}
The operator $(Y_0,D(Y_0))$ is affiliated to $M$ because the action of any $w'\in M'$ leaves globally invariant the domain $M'\xi_0$ and $w'Y_0(y'\xi_0)=w'y'\xi=Y_0(w'y'\xi_0)$ for any $y'\in M'$. The operator $(Y_0,D(Y_0))$ is closable because it is in duality with the densely defined operator $Z_0:M'\xi_0\to L^2(M)$ given by $Z_0(z'\xi_0):=z'S_0\xi$ in the sense that\\
\[
\begin{split}
(z'\xi_0|Y_0(y'\xi_0))&=(z'\xi_0|y'\xi)=(y'^*z'\xi_0|\xi)=(J\Delta_0^{-1/2}z'^*y'\xi_0|\xi)\\
&=(J\xi|\Delta_0^{-1/2}z'^*y'\xi_0)=(z'J\Delta_0^{1/2}\xi|y'\xi_0)=(z'S_0\xi|y'\xi_0)\\
&=(Z_0(z'\xi_0)|y'\xi_0).
\end{split}
\]
Clearly by definition $Y_0\xi_0=\xi$ and the calculation above implies $Y_0^*\xi_0=S_0\xi$. If $(Y,D(Y))$ is a closed operator affiliated to $M$ with properties i) and ii) above, then, as $\xi_0\in D(Y)$, we have $y'\xi_0\in D(Y)$ for all $y'\in M'$ so that $M'\xi_0\subseteq D(Y)$ and $Y(y'\xi_0)=y'Y\xi_0=y'\xi=Y_0(y'\xi_0)$, which shows that $(Y,D(Y))$ is a closed extension of $({\overline{Y_0}}, D({\overline{Y_0}}))$.
\end{proof}
\par\noindent
This representation will be applied below to eigenvectors $\xi$ (if any) of the modular operator.

\begin{lem}
Let $(Y,D(Y))$ be a densely defined, closed operator affiliated to $M$ and $\mu,\nu\ge 0$. Then defining $d_Y^{\mu,\nu}:D(d_Y^{\mu,\nu})\to L^2(M)$ as
\[
d_Y^{\mu,\nu}:=i(\mu Y-\nu j(Y^*))\qquad D(d_Y^{\mu,\nu}):=D(Y)\cap JD(Y^*),
\]
it results that $(d_Y^{\mu,\nu},D(d_Y^{\mu,\nu}))$ is a densely defined, closable operator on $L^2(M)$.
\end{lem}
\begin{proof}
Since $J^2=I$, we have $D(j(Y^*))=JD(Y^*)$ so that $d_Y^{\mu,\nu}$ is
well defined on $D(d_Y^{\mu,\nu})$. By hypotheses, $j(Y^*)$ is densely defined, closed and affiliated to the commutant von Neumann algebra $M'$. Hence $Y$ and $j(Y^*)$ strongly commute and the contraction semigroup $e^{-t|Y|}\circ e^{-t|j(Y^*)|}=e^{-t|j(Y^*)|}\circ e^{-t|Y|}$ with parameter $t\ge 0$ strongly converges to the identity operator on $L^2(M)$ as $t\to 0^+$. Since $d_Y^{\mu,\nu}\circ e^{-t|Y|}\circ e^{-t|j(Y^*)|}=i(\mu Y\circ e^{-t|Y|}\circ e^{-t|j(Y^*)|}-\nu e^{-t|Y|} \circ j(Y^*)\circ e^{-t|j(Y^*)|})$ is a bounded operator for any $t>0$, we have that $D(d_Y^{\mu,\nu})$ is dense in $L^2(M)$.\\
To prove the statement concerning closability, observe that reasoning as above with $Y^*$ in place of $Y$ and $Y^{**}=Y$ in place of $Y^*$, we have that $\mu Y^*-\nu j(Y)$ is densely defined on $D(Y^*)\cap JD(Y)$. Moreover, since
\[
\begin{split}
(d_Y^{\mu,\nu}\eta|\zeta)&=(i(\mu Y-\nu j(Y^*))\eta|\zeta)\\
&=-i\mu (Y\eta|\zeta) +i\nu( j(Y^*)\eta|\zeta)\qquad \eta\in D(d_Y^{\mu,\nu}):=D(Y)\cap JD(Y^*)\\
&=-i\mu (\eta|Y^*\zeta) +i\nu(\eta| j(Y)\zeta)\qquad \zeta\in D(Y^*)\cap JD(Y)\\
&=(\eta|-i(\mu Y^*-\nu j(Y))\zeta),\\
\end{split}
\]
the adjoint of $(d_Y^{\mu,\nu}, D(d_Y^{\mu,\nu}))$ is an extension of $(-i(\mu Y^*-\nu j(Y)), D(Y^*)\cap JD(Y))$. It is thus densely defined and consequently
$(d_Y^{\mu,\nu}, D(d_Y^{\mu,\nu}))$ is closable.
\end{proof}
\begin{lem}
Let $(Y,D(Y))$ be a densely defined, closed operator affiliated to $M$. Then the $J$-real part of the domain $D(Y)$ is invariant under the modulus
map:
\[
\xi\in D(Y)\, ,\quad J\xi=\xi\qquad \Rightarrow\qquad |\xi|\in D(Y)
\]
and $\|Y|\xi|\|=\|Y\xi\|$. In particular, if $\xi=\xi_+-\xi_-$ is the
polar decomposition of a $J$-real vector $\xi=J\xi\in D(Y)$, then $\xi_\pm=(|\xi|\pm\xi)/2\in D(Y)$ and
\[
\|Y\xi_\pm\|\le \|Y\xi\|\, .
\]
\end{lem}
\begin{proof}
Consider first the case where $Y$ is bounded, and let $s'_\pm\in M'$ be the supports in $M'$ of the positive and negative parts $\xi_\pm$ of a $J$-real $\xi\in L^2(M)$. Then
$(Y\xi_+|Y\xi_-)=(Ys'_+\xi_+|Ys'_-\xi_-)=(s'_+Y\xi_+|s'_-Y\xi_-)=0$
since $\xi_+\perp\xi_-$ imply $s'_+s'_-=0$, by \cite{ara1} Theorem 4. Thus
\[
\begin{split}
\|Y\xi\|^2 &= (Y\xi|Y\xi)=(Y\xi_+-Y\xi_-|Y\xi_+-Y\xi_-)\\
&=(Y\xi_++Y\xi_-|Y\xi_++Y\xi_-)=\|Y|\xi|\|^2.
\end{split}
\]
To deal with the general case, fix $\xi=J\xi\in D(Y)=D(|Y|)$ and consider the family of bounded operators $|Y|_\varepsilon :=|Y|(I+\varepsilon |Y|)^{-1}\in M$ for $\varepsilon>0$ as well as the spectral measure $E^{|Y|}$ of the self-adjoint operator $|Y|$. Applying the result concerning the bounded case, for all $\varepsilon >0$ we have
\[
\int_0^{+\infty} E^{|Y|}_{|\xi|,|\xi|}(d\lambda)\frac{\lambda^2}{(1+\varepsilon \lambda)^2}=\||Y|_\varepsilon |\xi|\|^2=\||Y|_\varepsilon\xi\|^2
=\int_0^{+\infty} E^{|Y|}_{\xi,\xi}(d\lambda)\frac{\lambda^2}{(1+\varepsilon \lambda)^2}.
\]
Letting $\varepsilon \downarrow 0$, by the Monotone Convergence Theorem we have $|\xi|\in D(|Y|)=D(Y)$ and $\|Y|\xi|\|=\|Y\xi\|$.
\end{proof}
\begin{lem}
Let $(Y,D(Y))$ be a densely defined, closed operator affiliated to $M$ and $\mu,\nu\ge 0$. Let $\xi\in D(d_Y^{\mu,\nu}):=D(Y)\cap JD(Y^*)$ be a $J$-real vector with polar decomposition $\xi=\xi_+-\xi_-$. Then $\xi_\pm\in D(Y)\cap D(Y^*)$ and
\[
(Y\xi_+|j(Y^*)\xi_-)\ge 0\, .
\]
\end{lem}
\begin{proof}
If $Y\in M$ the assertion is true because in that case $Y^*j(Y^*)$ is positivity preserving. To deal with the general case, let $Y=U|Y|$ be the polar decomposition of $Y$. By the previous Lemma 2.3, since $\xi=J\xi\in D(Y)$, we have $\xi_\pm\in D(Y)=D(|Y|)$. Since also $Y^*$ is a densely defined, closed operator affiliated to $M$ and, by assumption, $\xi=J\xi\in JJD(Y^*)=D(Y^*)$, again by Lemma 2.3 we have $\xi_-\in D(Y^*)$ too so that $(Y\xi_+|j(Y^*)\xi_-)=(|Y|\xi_+|j(|Y|)U^*j(U^*)\xi_-)$ (here we implicitly used the fact that $U^*j(U^*)\xi_-\in D(|Y|)$ since $\xi_-\in D(|Y|)$ by the previous Lemma 2.3, $U\in M$ and $j(|Y|)$ is affiliated to the commutant algebra $M'$). Since $U\in M$ so that $U^*j(U^*)$ is positivity preserving, it is enough to prove that $(|Y|\eta|j(|Y|)\eta')\ge 0$ for all $\eta\, , \eta'\in D(|Y|)\cap L^2_+(M)$. Since $|Y|$ and $j(|Y|)$ strongly commute, we can represent the value $(|Y|\eta|j(|Y|)\eta')$ as an integral over the product of the spectral measures of the two operators
\[
(|Y|\eta|j(|Y|)\eta')=\int_{[0,+\infty)^2}(E^{|Y|}_\eta\times E^{j(|Y|)}_{\eta'})(d\lambda,d\lambda')\lambda\cdot\lambda' \, .
\]
Setting $f_\varepsilon (\lambda):=\lambda/(1+\varepsilon\lambda)$ we have $0\le f_\varepsilon(\lambda)\le\lambda$ and $\lim_{\varepsilon\to 0}f_\varepsilon (\lambda)=\lambda$. By the Dominated Convergence Theorem we
have
\[
\begin{split}
(|Y|\eta|j(|Y|)\eta')&=\int_{[0,+\infty)^2}(E^{|Y|}_\eta\times E^{j(|Y|)}_{\eta'})(d\lambda,d\lambda')\lambda\cdot\lambda' \\
&=\lim_{\varepsilon\to 0} \int_{[0,+\infty)^2}(E^{|Y|}_\eta\times E^{j(|Y|)}_{\eta'})(d\lambda,d\lambda')f_\varepsilon(\lambda)\cdot f_\varepsilon(\lambda') \\
&=\lim_{\varepsilon\to 0} (f_\varepsilon(|Y|)\eta|j(f_\varepsilon (|Y|))\eta') \\
&=\lim_{\varepsilon\to 0} (\eta|f_\varepsilon(|Y|)j(f_\varepsilon(|Y|))\eta')\ge 0
\end{split}
\]
since $f_\varepsilon(|Y|)\in M$ and $f_\varepsilon(|Y|)j(f_\varepsilon(|Y|))$ is positivity preserving.
\end{proof}

\begin{thm}
Let $(Y,D(Y))$ be a densely defined, closed operator affiliated to $M$ such that $\xi_0\in D(Y)\cap D(Y^*)$ and $\mu,\nu>0$. Then the quadratic form $\tilde\E_Y^{\mu,\nu}:\tilde\F_Y\to [0,+\infty)$ on $L^2(M)$
\[
\tilde\E_Y^{\mu,\nu}[\xi]:=\|d_Y^{\mu,\nu}\xi\|^2_{L^2(M)} + \|d_{Y^*}^{\nu,\mu}\xi\|^2_{L^2(M)}\qquad \tilde\F_Y:=D(d_Y^{\mu,\nu})\cap D(d_{Y^*}^{\nu,\mu})
\]
is densely defined, closable, $J$-real (recall that $\|d_{Y^*}^{\nu,\mu}\xi\|^2_{L^2(M)}=\|d_{Y}^{\mu,\nu}J\xi\|^2_{L^2(M)}$) and satisfies the first Beurling-Deny condition
\[
\xi=J\xi\in \tilde\F_Y\qquad\Rightarrow\qquad \xi_\pm\in\tilde\F_Y\qquad\text{and}\qquad \tilde\E_Y^{\mu,\nu}(\xi_+|\xi_-)\le 0.
\]
%equivalently expressed as
%\[
%\xi=J\xi\in\tilde\F_Y\quad \Rightarrow\quad |\xi|\in\tilde\F_Y\qquad\text{and}\qquad \tilde\E[|\xi|]\le \tilde\E[\xi].
%\]
%\noindent
Its closure $(\E_Y^{\mu,\nu},\F_Y^{\mu,\nu})$ satisfies the first Beurling-Deny condition too and generates a contractive, positivity preserving semigroup $\{T_t:t\ge 0\}$.
Moreover, $(\E_Y^{\mu,\nu},\F_Y^{\mu,\nu})$ is a conservative, in the sense that
\[
\xi_0\in\F_Y^{\mu,\nu},\qquad \E_Y^{\mu,\nu}[\xi_0]=0,
\]
completely Dirichlet form with respect to $(M,\o_0)$ and the associated completely Markovian semigroup is conservative, in the sense that
\[
T_t\xi_0=\xi_0\qquad t\ge 0,
\]
if and only if $Y\xi_0\in L^2(M)$ is an eigenvector of the modular operator corresponding to the eigenvalue $\mu/\nu$
\[
Y\xi_0\in D(\Delta_0^{1/2}),\qquad \Delta_0^{1/2}Y\xi_0=(\mu/\nu)Y\xi_0.
\]
%{\color{red}and in this case} $S_0(Y\xi_0)=Y^*\xi_0$.
\end{thm}

\begin{proof}
Since $\tilde\F_Y=D(Y)\cap JD(Y^*)\cap D(Y^*)\cap JD(Y)=D(Y)\cap D(Y^*)\cap J(D(Y)\cap D(Y^*))$, we have $J\tilde\F_Y=\tilde\F_Y$ and $\xi_0=J\xi_0\in\tilde\F_Y$. Since $d_Y^{\mu,\nu}J\xi=i(\mu Y-\nu JY^*J)J\xi=i(\mu YJ\xi-\nu JY^*\xi)=iJ(\mu JYJ\xi-\nu Y^*\xi)=Ji(\nu Y^*\xi-\mu JYJ\xi)=Jd_{Y^*}^{\nu,\mu}\xi$ for all $\xi\in\tilde\F_Y$ and, exchanging the role of $Y$ and $Y^*$, we have $d_{Y^*}^{\nu,\mu}J\xi=Jd_Y^{\mu,\nu}\xi$ too, we get
\[
\tilde\E_Y^{\mu,\nu}[J\xi]=\|d_Y^{\mu,\nu}J\xi\|^2 + \|d_{Y^*}^{\nu,\mu}J\xi\|^2=\|Jd_{Y^*}^{\nu,\mu}\xi\|^2 + \|Jd_Y^{\mu,\nu}\xi\|^2=\tilde\E_Y^{\mu,\nu}[\xi]\qquad \xi\in \tilde\F_Y\, ,
\]
which proves that the quadratic form $(\tilde\E_Y^{\mu,\nu},\tilde\F_Y)$ is $J$-real.\\
Consider now a $J$-real vector $\xi\in\tilde\F_Y$, its polar decomposition $\xi=\xi_+-\xi_-$ with respect to the self-polar cone $L^2_+(M)$ and recall that, by definition, $|\xi|:=\xi_+ + \xi_-$. By the previous lemma, $|\xi|\in\tilde\F_Y$ so that $\xi_\pm=(|\xi|\pm\xi)/2\in\tilde\F_Y$. Then, if $s_\pm\in M$ (resp. $s'_\pm\in M'$) are the supports of $\xi_\pm$ in $M$ (resp. $M'$), we have
\[
\begin{split}
(d_Y^{\mu,\nu}\xi_+|d_Y^{\mu,\nu}\xi_-)=&((\mu Y\xi_+-\nu j(Y^*)\xi_+|(\mu Y\xi_- -\nu j(Y^*)\xi_-) \\
=&\mu^2(Ys'_+\xi_+| Ys'_-\xi_-) +\nu^2(j(Y^*)s_+\xi_+|j(Y^*)s_-\xi_-) \\
&-\mu\nu\Bigl((Y\xi_+|j(Y^*)\xi_-) + (j(Y^*)\xi_+|Y\xi_-)\Bigr) \\
=&-\mu\nu\Bigl((Y\xi_+|j(Y^*)\xi_-) + (j(Y^*)\xi_+|Y\xi_-)\Bigr)\le 0
\end{split}
\]
by Lemma 2.4. Since, analogously, $(d_{Y^*}^{\nu,\mu}\xi_+|d_{Y^*}^{\nu,\mu}\xi_-)\le 0$ we have $\tilde\E_Y^{\mu,\nu}(\xi_+|\xi_-)\le 0$ and consequently the first Beurling-Deny condition is satisfied by $(\tilde\E_Y^{\mu,\nu},\tilde\F_Y)$
\[
\xi=J\xi\in\tilde\F_Y\qquad\Rightarrow\qquad |\xi|\in\tilde\F_Y\qquad\text{and}\qquad \tilde\E_Y^{\mu,\nu}[|\xi|]\le \tilde\E_Y^{\mu,\nu}[\xi].
\]
To establish the same condition for the closure $(\E_Y^{\mu,\nu},\F_Y^{\mu,\nu})$, we adapt the proof of \cite{c1} Proposition 5.1 (according the suggestions which there precede it).\\
On one hand, since $J$ is an isometry for the graph norm of $(\tilde\E_Y^{\mu,\nu},\tilde\F_Y)$ and $\tilde\F_Y$ is a core for $(\E_Y^{\mu,\nu},\F_Y^{\mu,\nu})$,then $J$ is an isometry for the latter form too and the closure form is $J$-real.\\
On the other hand, let $\xi=J\xi\in\F_Y^{\mu,\nu}$ be a fixed $J$-real vector and let $\xi_n\in\tilde\F_Y$ be a sequence converging to it in the graph norm of $(\E_Y^{\mu,\nu},\F_Y^{\mu,\nu})$. Since the Hilbert projection $\eta\mapsto\eta_+$ of the $J$-real part of $L^2(M)$ onto the closed, convex cone $L^2_+(M)$ is norm continuous and $|\eta|=2\eta_+-\eta$, it follows that the modulus map $\eta\mapsto |\eta|$ is norm continuous too. Then, since the form $\E_Y^{\mu,\nu}$ is norm lower semicontinuous on $L^2(M)$, it follows that
\[
\begin{split}
\E_Y^{\mu,\nu}[|\xi|]&\le\liminf_n\E_Y^{\mu,\nu}[|\xi_n|]=\liminf_n\tilde\E_Y^{\mu,\nu}[|\xi_n|]
\le\liminf_n\tilde\E_Y^{\mu,\nu}[\xi_n] \\
&=\liminf_n\E_Y^{\mu,\nu}[\xi_n]=\E_Y^{\mu,\nu}[\xi].
\end{split}
\]
The first Beurling-Deny condition is thus verified and, by \cite{c1} Proposition 4.10, it follows that the semigroup $\{T_t:t\ge 0\}$ has the desired properties.\\
Concerning the conservativeness property, notice that  $\xi_0\in\tilde\F_Y\subseteq\F_Y^{\mu,\nu}$. If $\E_Y^{\mu,\nu}[\xi_0]=0$ then $\tilde\E_Y[\xi_0]=0$ so that $d_Y^{\mu,\nu}\xi_0=0$ which implies
$\mu Y\xi_0=\nu j(Y^*)\xi_0=\nu J Y^*\xi_0$ and, for any $x\in M$, $\mu(x\xi_0|Y\xi_0)=\nu (x\xi_0|J Y^*\xi_0)=\nu (Y^*\xi_0|J x\xi_0)=\nu (Y^*\xi_0|J xJ\xi_0)=\nu (J x^*J\xi_0|Y\xi_0)=\nu (J x^*\xi_0|Y\xi_0)=\\\nu (\Delta_0^{1/2}x\xi_0|Y\xi_0)$ since $Y$ is affiliated to $M$ and $Jx^*J\in M'$. Setting $\lambda^2:=\mu/\nu$, this in turn implies $((\Delta_0^{1/2}-\lambda^2 I)x\xi_0|Y\xi_0)=0$ for all $x\in M$ and, since $M\xi_0$ is a core for $\Delta_0^{1/2}$, it follows that $Y\xi_0\in D(\Delta_0^{1/2})$ and $\Delta_0^{1/2} Y\xi_0=(\mu/\nu) Y\xi_0$, i.e. $Y\xi_0$ is an eigenvalue of the modular operator with eigenvalue $\mu/\nu$.\\
On the other hand, if $Y\xi_0\in D(\Delta_0^{1/2})$ and $\Delta_0^{1/2}Y\xi_0=(\mu/\nu)Y\xi_0$, using the identities above, it follows that $d_Y^{\mu,\nu}\xi_0=0$  and $d_{Y^*}^{\nu,\mu}\xi_0=d_Y^{\mu,\nu}J\xi_0=d_Y^{\mu,\nu}\xi_0=0$ so that $\E_Y^{\mu,\nu}\xi_0=\tilde\E_Y[\xi_0]=0$, i.e. $(\E_Y^{\mu,\nu},\F_Y^{\mu,\nu})$ is conservative. By \cite{c1} Proposition 4.10, given conservativness, the first Beurling-Deny property and Markovianity are equivalent for quadratic forms as well as the positivity preserving property and the Markovianity are equivalent for the associated semigroups.\\
Concerning the complete Markovianity of the Dirichlet form, we notice that for any $n\ge 1$, the ampliation $(\E^\lambda_Y)^n:L^2(M\otimes M_n(\mathbb{C}),\o_0\otimes {\rm tr}_n)\to [0,+\infty]$, defined as
$(\E^\lambda_Y)^n[[\xi_{i,j}]_{i,j=1}^n]:=\sum_{i,j=1}^n\E[\xi_{i,j}]$, has the same structure as $\E^\lambda_Y$. More precisely, a closed operator $Y^n:=Y\otimes I_n$ is densely defined on $D(Y^n):=D(Y)\otimes_{\rm alg}L^2(M_n(\mathbb{C}),tr_n)\subset L^2(M\otimes M_n(\mathbb{C}),\o_0\otimes tr_n)$ and one may check that $(\E^\lambda_Y)^n=\E^\lambda_{Y^n}$. If $\xi_0\in D(Y)\cap D(Y^*)$ and $Y\xi_0\in L^2(M)$ is an eigenvector of the modular operator $\Delta_0^{1/2}$ of the state $\o_0$ on $M$, corresponding to the eigenvalue $\mu/\nu$, then, denoting by $\zeta_n\in L^2(M_n(\mathbb{C}),tr_n)$ the unit vector representing the trace state, it easily verified that $\xi_0\otimes \zeta_n\in D(Y^n)\cap D((Y^n)^*)$ and that $Y^n(\xi_0\otimes\zeta_n)=Y\xi_0\otimes\zeta_n$ is an eigenvalue of the modular operator of the state $\o_0\otimes tr_n$ on $M\otimes M_n(\mathbb{C})$, corresponding to the same eigenvalue $\mu/\nu$. Applying the results obtained above to the form $(\E^\lambda_Y)^n$ in place of $\E^\lambda_Y$, we get its Markovianity for any $n\ge 1$ and complete Markovianity of the associated semigroup.
\end{proof}
\noindent
{\bf Notation.} If $\lambda^2\in {\rm Sp}(\Delta_0^{1/2})\setminus\{0\}$ is a strictly positive eigenvalue of the modular operator and $\mu/\nu=\lambda^2$, then
$d_Y^{\mu,\nu}=\sqrt{\mu\nu} d_Y^{\lambda,\lambda^{-1}}$, $d_{Y^*}^{\nu,\mu}=\sqrt{\mu\nu} d_{Y^*}^{\lambda^{-1},\lambda}$ and $\E_Y^{\mu,\nu}=\mu\nu\cdot \E_Y^{\lambda,\lambda^{-1}}$. Since now on, we will adopt the simplified notation
\[
\E_Y^\lambda:=\E_Y^{\lambda,\lambda^{-1}}.
\]
{\bf Remark.} To any eigenvector $\xi\in D(S_0)$ of $\Delta_0^{1/2}$ or, equivalently, of the Araki Hamiltonian $\ln\Delta_0$, we associate a completely Dirichlet form $\E^\lambda_Y$ choosing a densely defined, closed operator $(Y,D(Y))$ as in Lemma 2.1. For this choice there exists a canonical candidate, namely $({\overline{Y_0}}, D({\overline{Y_0}}))$. In general $(\E^\lambda_Y, \F^\lambda_Y)$ may depend upon the operator $(Y,D(Y))$ and not only upon the eigenvector $\xi=Y\xi_0$ it represents. The next result shows how this is connected to the GNS symmetry of the Markovian semigroup.
\begin{thm}(GNS symmetry)
Let $(Y,D(Y))$ be a densely defined, closed operator affiliated to $M$, $\mu,\nu>0$ such that $\xi_0\in D(Y)\cap D(Y^*)$ and $Y\xi_0\in L^2(M)$ is an eigenvector of $(\Delta_0^{1/2},D(\Delta_0^{1/2}))$  for the eigenvalue $\lambda^2:=\mu/\nu$. Then, for any $t\in\R$,
\vskip0.1truecm\noindent
i) the densely defined, closed operator $(Y_t,D(Y_t)):=(\Delta_0^{it}Y\Delta_0^{-it},\Delta_0^{it}D(Y))$, affiliated to $M$,  verifies $\xi_0\in D(Y_t)\cap D(Y_t^*)$,
$Y_t\xi_0=\lambda^{4ti}Y\xi_0\in L^2(M)$ and $Y_t=\lambda^{4ti}Y$ on the subspace $M'\xi_0$;\\
ii) $(\E^\lambda_{Y_t},\F^\lambda_{Y_t})$ is a Dirichlet form with respect to $(M,\xi_0)$ coinciding with
\[
\F^\lambda_{Y_t}=\Delta_0^{it}(\F^\lambda_Y)\qquad \E^\lambda_{Y_t}[\eta]=\E^\lambda_Y[\Delta_0^{-it}\eta].
\]
If, moreover, $M'\xi_0\subseteq D(Y)$ is a core for $(Y,D(Y))$, then, for any $t\in\R$, we have\\
iii) $(Y_t,D(Y_t))=(\lambda^{4it}\cdot Y,D(Y))$, for any $t\in\R$;\\
iv) $(\E^\lambda_{Y_t},\F^\lambda_{Y_t})=(\E^\lambda_Y,\F^\lambda_Y)$, the associated Markovian semigroup is symmetric
\[
(T_t(x\xi_0)|y\xi_0)=(x\xi_0|T_t(y\xi_0))\qquad x,y\in M,\,\, t\ge 0
\]
and, in particular, it commutes with $\{\Delta_0^{it}:t\in\R\}$;\\
v) The semigroup generated by $(\E^\lambda_{\overline Y_0},\F^\lambda_{\overline Y_0})$ is GNS symmetric (notations of Lemma 2.1).
\end{thm}
\begin{proof}
i) Since, for any $t\in\R$, one has $\Delta_0^{it}\xi_0=\xi_0$, it follows that $\xi_0\in D(Y_t)\cap D(Y_t^*)$, $Y_t\xi_0=
\Delta_0^{it}Y\xi_0=\lambda^{4ti}Y\xi_0\in L^2(M)$ and $Y_t(z'\xi_0)=z'Y_t\xi_0=\lambda^{4ti}\cdot z'Y\xi_0=\lambda^{4ti}\cdot Y (z'\xi_0)$ for any $z'\in M'$; ii) thus $Y_t\xi_0$ is an eigenvector of $(\Delta_0^{1/2},D(\Delta_0^{1/2}))$ for the eigenvalue $\lambda^2$ and, by Theorem 2.5,  $(\E^\lambda_{Y_t},\F^\lambda_{Y_t})$ is a well defined Dirichlet form. The displayed identity follows from the identities $d^{\mu,\nu}_{Y_t}=\Delta_0^{it}\circ d^{\mu,\nu}_Y\circ \Delta_0^{-it}$, $d^{\mu,\nu}_{Y_t^*}=\Delta_0^{it}\circ d^{\mu,\nu}_{Y^*}\circ \Delta_0^{-it}$, valid, for any $t\in\R$, on $\tilde \F_{Y_t}=\tilde \F_Y$ and the fact that this space is a form core for $\eta\mapsto \E^\lambda_{Y_t}[\eta]$ and $\eta\mapsto\E^\lambda_Y[\Delta_0^{-it}\eta]$.\\
iii) Since the core $M'\xi_0$ for $(Y,D(Y))$ is invariant under the group $\{\Delta_0^{it}:t\in\R\}$, it is a core also for $(Y_t,D(Y_t))$, for any fixed $t\in\R$. Since, by i), $Y_t=\lambda^{4ti}\cdot Y$ on
this common core, we have $D(Y_t)=D(Y)$ and $Y_t=\lambda^{4ti}\cdot Y$, for any $t\in\R$; iv) since, by iii), $(\E^\lambda_{Y_t},\F^\lambda_{Y_t})=(\E^\lambda_Y,\F^\lambda_Y)$ for any $t\in\R$, ii) implies that $(\E^\lambda_Y,\F^\lambda_Y)$ is invariant under the unitary group $\{\Delta_0^{it}:t\in\R\}$ so that the Markovian semigroup it generates commutes with $\{\Delta_0^{it}:t\in\R\}$ and it is GNS symmetric by \cite{c1} Theorem 6.6; v) follows from iv) as, by definition, $M'\xi_0$ is a core for $(\overline{Y_0},D(\overline{Y_0}))$.
\end{proof}

\subsection{Representation of Dirichlet forms as square of commutators}
In this section we show how to represent the Dirichlet forms on $L^2(M)$ constructed above, in terms of generalized commutators, i.e. unbounded spatial derivations on $M$.\\
We recall that $(S_0,D(S_0))$ is an {\it unbounded} conjugation, i.e. anti-linear and idempotent on its domain. Thus $S_0^2$ is the identity operator on $D(S_0)$ or, more explicitly, that $\xi\in D(S_0)$ implies $S_0\xi\in D(S_0)$ and $S_0(S_0\xi)=\xi$. In other terms, the image of $S_0$ coincides with its domain and $S_0=S_0^{-1}$ holds true as an identity between densely defined, closed operators. In terms of the polar decomposition $S_0=J\Delta_0^{1/2}$ we have $J\Delta_0^{1/2}=\Delta_0^{-1/2}J$ as an identity between densely defined, closed operators. This means, in particular, that the modular conjugation exchanges domains as follows $JD(\Delta_0^{1/2})=D(\Delta_0^{-1/2})$, $D(\Delta_0^{1/2})=JD(\Delta_0^{-1/2})$. More in general, one has the intertwining relation ${\bar f}(\Delta_0^{-1})=Jf(\Delta_0) J$ between closed operators valid for any Borel measurable function $f:[0,+\infty)\to\C$ (see Introduction to Chapter 10 in \cite{sz}). The relation, which  is equivalent to $JD({\bar f}(\Delta_0^{-1}))=D(f(\Delta_0) )$ and ${\bar f}(\Delta_0^{-1})\xi=Jf(\Delta_0) J\xi$ for all $\xi \in D({\bar f}(\Delta_0^{-1}))$, will be mostly used for power functions $f$.
\vskip0.1truecm\noindent
Among its consequences, we will make use of the following:
\vskip0.1truecm\noindent
a)  for any $\alpha\in\R$, the closed operator $J\Delta_0^{\alpha}$ is an unbounded conjugation on its domain $D(\Delta_0^{\alpha})$;
\vskip0.1truecm\noindent
b) $S_0=J\Delta_0^{1/2}=\Delta_0^{-1/4}J\Delta_0^{1/4}$ is an identity between densely defined, closed operators: in fact, $D(\Delta_0^{-1/4}J\Delta_0^{1/4}):=\{\xi\in D(\Delta_0^{1/4}):J\Delta_0^{1/4}\xi\in D(\Delta_0^{-1/4})\}$ but since $D(\Delta_0^{-1/4})=JD(\Delta_0^{1/4})$ one has $D(\Delta_0^{-1/4}J\Delta_0^{1/4})=\{\xi\in D(\Delta_0^{1/4}):\Delta_0^{1/4}\xi\in D(\Delta_0^{1/4})\}=D(\Delta_0^{1/2})$ and, for all $\xi\in D(\Delta_0^{1/2})$,
\[
(\Delta_0^{-1/4}J\Delta_0^{1/4})\xi=(\Delta_0^{-1/4}J)\Delta_0^{1/4}\xi=(J\Delta_0^{1/4})\Delta_0^{1/4}\xi=J\Delta_0^{1/2}\xi;
\]
c) $(J\Delta_0^{1/4},D(\Delta_0^{1/4}))$ is a closed extension of the densely defined operator $(\Delta_0^{1/4}S_0,D(S_0))$: in fact, the latter operator is well defined since $\zeta\in D(S_0)$ implies $S_0\zeta\in D(S_0)=D(\Delta_0^{1/2})\subset D(\Delta_0^{1/4})$ and also $\Delta_0^{1/4}S_0\zeta=\Delta_0^{1/4}S_0^{-1}\zeta=\Delta_0^{1/4}\Delta_0^{-1/2}J\zeta=\Delta_0^{-1/4}J\zeta=J\Delta_0^{1/4}\zeta$;
\vskip0.1truecm\noindent
d) $(\Delta_0^{-1/4}J,D(\Delta_0^{1/4}))=(J\Delta_0^{1/4},D(\Delta_0^{1/4}))$ is a closed extension of the densely defined operator $(S_0\Delta_0^{-1/4},D(\Delta_0^{-1/4})\cap D(\Delta_0^{1/4}))$: in fact $D(S_0\Delta_0^{-1/4}):=\{\zeta\in D(\Delta_0^{-1/4}): \Delta_0^{-1/4}\zeta\in D(S_0)\}=\{\zeta\in D(\Delta_0^{-1/4}): \Delta_0^{-1/4}\zeta\in D(\Delta_0^{1/2})\}=D(\Delta_0^{-1/4})\cap D(\Delta_0^{1/4})$ and $J\Delta_0^{1/4}\zeta=J\Delta_0^{1/2}\Delta_0^{-1/4}\zeta=S_0\Delta_0^{-1/4}\zeta$ for all $\zeta\in D(\Delta_0^{-1/4})\cap D(\Delta_0^{1/4})$;\\
\vskip0.1truecm\noindent
e) Let $M_0\subseteq M$ be  the involutive w$^*$-dense sub-algebra of analytic vectors of the group $\sigma^{\o_0}$. For any $y\in M_0$, the operator $\Delta_0^{1/4}y\Delta_0^{-1/4}$ on $L^2(M)$ is densely defined on $i_0(M_0)$ and closable. Its closure is a bounded operator belonging to $M$, which coincides with the analytic extension of the map $\R\ni t\mapsto \sigma^{\o_0}_t(y)\in M_0\subset M$ evaluated at $t=-i/4$
\[
\overline{(\Delta_0^{1/4}y\Delta_0^{-1/4})}=\sigma^{\o_0}_{-i/4}(y)\in M_0\subset M
\]
and $\sigma^{\o_0}_{-i/4}(y)i_0(x)=i_0(yx)$ for all $x\in M_0$;\\
f) by Proposition in Section 9.24 in \cite{sz}, for any $y\in M_0$ and any $\alpha\in\C$ one has the important identity
\[
D(\Delta_0^{\alpha}y\Delta_0^{-\alpha})=D(\Delta_0^{-\alpha})
\]
and the boundedness of the operator $\Delta_0^{\alpha}y\Delta_0^{-\alpha}$ on $D(\Delta_0^{-\alpha})$. Since $JD(\Delta_0^{1/2})=D(\Delta_0^{-1/2})$, the case $\alpha=1/2$ implies that $D(S_0yS_0)=D(S_0)$ and the boundedness of the operator $S_0yS_0$ on $D(S_0)$;\\
\noindent
g) the involutive sub-algebra  $M_0':=JM_0J\subset M'$ coincides with the set of analytic vectors of the modular group of the commutant $M'$ associated to the state determined by $\xi_0\in L^2(M)$. The left Hilbert sub-algebra $M_0\xi_0\subset M\xi_0\subset L^2(M)$ is dense in $L^2(M)$ and it coincides with the symmetric embedding of the algebra of analytic elements
\[
M_0\xi_0=i_0(M_0),
\]
as it results from the identity $i_0(y)=\sigma^{\o_0}_{-i/4}(y)\xi_0$ valid for all $y\in M_0$ . Also, $M_0\xi_0$ is $J$-invariant
\[
Ji_0(y)=i_0(y^*)\qquad y\in M_0.
\]
 %In this section we fix a vector $\xi\in D(S_0)$ and a densely defined, closed operator $(X,D(X))$ affiliated to $M$ such that
%\[
%\xi_0\in D(X)\cap D(X^*),\qquad X\xi_0{\color{red}=\xi},\quad S_0(X\xi_0)=X^*\xi_0.
%\]
%Recall that {\color{red} $D(\Delta_0^{1/4}y\Delta_0^{-1/4})=D(\Delta_0^{1/4})$ for all $y\in M$ (see [Stratila-Zsido Proposition 9.24]).}

\begin{lem}
If $\eta\in D(S_0)$, the densely defined operator $(\L_\eta,D(\L_\eta))$ given by
\[
D(\L_\eta):=i_0(M_0)\ni i_0(y)\qquad \L_\eta i_0(y):=J\sigma^{\o_0}_{-i/4}(y^*)J\eta
\]
is closable since its adjoint is an extension of the densely defined operator $B:D(B)\to L^2(M)$
\[
D(B):=M'\xi_0\ni z'\xi_0\qquad B(z'\xi_0):=z'S_0\eta.
\]
The densely defined operator $(\RR_\eta,D(\RR_\eta))$ given by
\[
D(\RR_\eta):=i_0(M_0)\ni i_0(y)\qquad \RR_\eta i_0(y):=\sigma^{\o_0}_{-i/4}(y)J\eta
\]
satisfies the relation $\RR_\eta=J\L_\eta J$  from which it follows that it is closable too.
\end{lem}
\begin{proof}
%Notice first that, if $y\in M_0$, since $M_0$ is involutive, we have $y^*\in M_0$ so that $\overline{(\Delta_0^{1/4}y^*\Delta_0^{-1/4})}\in M$, $J\overline{(\Delta_0^{1/4}y^*\Delta_0^{-1/4})}J\in M'$ and
%\[
%J\overline{(\Delta_0^{1/4}y^*\Delta_0^{-1/4})}J\xi_0=J\Delta_0^{1/4}y^*\Delta_0^{-1/4}J\xi_0=J\Delta_0^{1/4}y^*\xi_0=J\Delta_0^{1/4}J\Delta_0^{1/2}y\xi_0=\Delta_0^{1/4}y\xi_0=i_0(y).
%\]
%To symplify notations set $A:=\overline{(\Delta_0^{1/4}y^*\Delta_0^{-1/4})}$.
Since $\xi_0=J\xi_0\in D(\Delta_0^{-1/4})=D(\Delta_0^{1/4}y\Delta_0^{-1/4})$, $\Delta_0^{1/4}J=J\Delta_0^{-1/4}$, $y\xi_0\in D(\Delta_0^{1/2})$ and $\Delta_0^{1/2}y\xi_0\in D(\Delta_0^{-1/4})$, we have
\[
\begin{split}
J\sigma^{\o_0}_{-i/4}(y^*)J\xi_0&=J(\Delta_0^{1/4}y^*\Delta_0^{-1/4})\xi_0=J\Delta_0^{1/4}y^*\xi_0 \\
&=J\Delta_0^{1/4}J\Delta_0^{1/2}y\xi_0=\Delta_0^{1/4}y\xi_0=i_0(y).
\end{split}
\]
Since, moreover, $w'^*\xi_0=\Delta_0^{1/2}Jw'\xi_0$ for all $w\in M'$ and $J\sigma^{\o_0}_{-i/4}(y^*)J\in M'$ for all $y\in M_0$, for $z'\in M'$ we have
\[
\begin{split}
(z'\xi_0|\L_\eta i_0(y))&=(z'\xi_0|J\sigma^{\o_0}_{-i/4}(y^*)J\eta)=((J\sigma^{\o_0}_{-i/4}(y^*)J)^*z'\xi_0|\eta) \\
&=(\Delta_0^{1/2}Jz'^*J\sigma^{\o_0}_{-i/4}(y^*)J\xi_0|\eta)=(Jz'^*i_0(y)|\Delta_0^{1/2}\eta)\\
&=(J\Delta_0^{1/2}\eta |z'^*i_0(y)) =(z'S_0\eta |i_0(y))=(B(z'\xi_0)|i_0(y)).
\end{split}
\]
The relation between the operators $\L_\eta$ and $\RR_\eta$ follows from the identities $i_0(y^*)=\Delta_0^{1/4}y^*\xi_0=\Delta_0^{1/4}S_0(y\xi_0)
=\Delta_0^{1/4}\Delta_0^{-1/2}J(y\xi_0)=\Delta_0^{-1/4}J(y\xi_0)=J\Delta_0^{1/4}y\xi_0=Ji_0(y)$ for all $y\in M_0$ and the fact that
$J$ is idempotent and it leaves $D(\L_\eta)=D(\RR_\eta)=i_0(M_0)$ globally invariant: $J\L_\eta Ji_0(y)=J\L_\eta i_0(y^*)=JJ\sigma^\o_{-i/4}(y)J\eta=\RR_\eta i_0(y)$.
\end{proof}

\begin{lem}
Let $\xi\in D(S_0)$ and fix, by Lemma 2.1, a densely defined, closed operator $(X,D(X))$ affiliated to $M$ such that
\[
\xi_0\in D(X)\cap D(X^*),\qquad \xi=X\xi_0,\quad S_0(X\xi_0)=X^*\xi_0.
\]
Then the following properties hold true:
\vskip0.2truecm\noindent
i) the intersection of domains $D(X)\cap D(X^*)$ contains $M_0\xi_0$;
\vskip0.2truecm\noindent
ii) the images of $M_0\xi_0$ under $(X,D(X))$ and $(X^*,D(X^*))$ are contained in $D(S_0)$
\[
X(y\xi_0)\in D(S_0),\quad X^*(y^*\xi_0)\in D(S_0),\qquad \text{for all}\quad y\in M_0
\]
and
\[
S_0(Xy\xi_0)=y^*X^*\xi_0,\quad S_0(X^*y^*\xi_0)=yX\xi_0\qquad \text{for all}\quad y\in M_0;
\]
Consider the densely defined operators on $L^2(M)$ given by
\[
L_\xi i_0(y):=J(\Delta_0^{1/4}y^*\Delta_0^{-1/4})J\Delta_0^{1/4}\xi \qquad i_0(y)\in i_0(M_0)=:D(L_\xi ),
\]
\[
R_\xi i_0(y):=(\Delta_0^{1/4}y\Delta_0^{-1/4})J\Delta_0^{1/4}S_0(\xi) \qquad i_0(y)\in i_0(M_0)=:D(R_\xi ).
\]
These are closable, by Lemma 2.7, since $L_\xi=\L_\eta$ and $R_\xi=\RR_\eta$ for $\eta:=\Delta_0^{1/4}\xi\in D(S_0)$ and
\vskip0.2truecm\noindent
iii) for any $y\in M_0$ we have
\[
Xy\xi_0\in D(\Delta_0^{1/4}),\qquad L_\xi i_0(y)=\Delta_0^{1/4}Xy\xi_0;
\]
iv) for any $y\in M_0$ we have
\[
yX\xi_0\in D(\Delta_0^{1/4}),\qquad R_\xi i_0(y)=\Delta_0^{1/4}yX\xi_0;
\]
v) $\overline{L_\xi}$ is affiliated with $M$, $\overline{R_\xi}$ is affiliated with $M'$ and $\overline{R_\xi}=J\overline{L_\xi} J$;\\
vi) the operator $\Delta_0^{1/4}X\Delta_0^{-1/4}$ is well defined on $i_0(M_0)$ and there it coincides with $L_\xi$;\\
vii) the operator $J\Delta_0^{1/4}X^*\Delta_0^{-1/4}J$ is well defined on
$i_0(M_0)$ and there it coincides with $R_\xi$;\\
viii) If $\xi=X\xi_0$ is an eigenvector of $\Delta_0^{1/2}$ corresponding to the eigenvalue $\lambda^2>0$, with $\lambda>0$,
\[
\Delta_0^{1/2}X\xi_0=\lambda^2\cdot X\xi_0,
\]
then $L_\xi=\lambda X$ and $R_\xi=\lambda^{-1} JX^*J$ on $i_0(M_0)$.
\end{lem}
\begin{proof}
i) As $X$ and $X^*$ are affiliated to $M$ and $\xi_0\in D(X)\cap D(X^*)$, it follows that $M'\xi_0\subset D(X)\cap D(X^*)$ and, a fortiori, that $M_0\xi_0=M'_0\xi_0\subset M'\xi_0\subset D(X)\cap D(X^*)$.\\
ii) Since $M_0\xi_0$ is a core for $(S_0,D(S_0))$, there exists a
sequence $x_n\in M_0$ such that $\|x_n\xi_0-X\xi_0\|\to 0$ and $\|x_n^*\xi_0-X^*\xi_0\|\to 0$. As mentioned at item f) of the introduction of the present section, since $y\in M_0$, the operator $S_0y^*S_0
%=J\Delta^{1/2}y^*\Delta_0^{-1/2}J
$ is bounded on $D(S_0)$ and then on $M_0\xi_0\subset D(S_0)$. Thus $x_ny\xi_0\in D(S_0)$ is a Cauchy sequence in $L^2(M)$ as
\[
\begin{split}
\|x_ny\xi_0-x_my\xi_0\|&=\|S_0(y^*x_n^*\xi_0-y^*x_m^*\xi_0)\|=\|S_0y^*S_0(x_n\xi_0-x_m\xi_0)\| \\
&\le \|S_0y^*S_0\|\cdot \|x_n\xi_0-x_m\xi_0\|.
\end{split}
\]
Analogously, $S_0(x_ny\xi_0)=y^*x_n^*\xi_0\in L^2(M)$ is a Cauchy sequence too as
\[
\|S_0(x_ny\xi_0)-S_0(x_my\xi_0)\|=\|y^*x_n^*\xi_0-y^*x_m^*\xi_0\|\le \|y^*\|\cdot \|x_n^*\xi_0-x_m^*\xi_0\|.
\]
Hence $x_ny\xi_0\in D(S_0)$ is a Cauchy sequence in the graph norm of the closed operator $(S_0,D(S_0))$ and the image of $\eta:=\lim_n x_n y\xi_0\in D(S_0)$ is given by  $S_0\eta=\lim_ny^*x_n^*\xi_0=y^*X^*\xi_0$. 
Since $X$ is affiliated to $M$, $X\xi_0-x_n\xi_0\in D(S_0)$ and $S_0y^*S_0$ is bounded on $D(S_0)$, for $z'\in M'$ we have
\[
\begin{split}
(z'\xi_0|Xy\xi_0-x_ny\xi_0)&=(X^*\xi_0-x_n^*\xi_0|yz'^*\xi_0)=(y^*S_0(X\xi_0-x_n\xi_0)|S_0^*(z'\xi_0))\\
&=((S_0y^*S_0)(X\xi_0-x_n\xi_0)|z'\xi_0))\\
&\le \|S_0y^*S_0\|\cdot \|X\xi_0-x_n\xi_0\|\cdot \|z'\xi_0\|.
\end{split}
\]
By the density of $M'\xi_0$ in $L^2(M)$, it follows that $\|Xy\xi_0-x_ny\xi_0\|\le \|S_0y^*S_0\|\cdot \|X\xi_0-x_n\xi_0\|\to 0$ as $n\to\infty$ and we have $Xy\xi_0=\eta\in D(S_0)$ and $S_0(Xy\xi_0)=S_0(\eta)=y^*X^*\xi_0$ for any $y\in M_0$.\\
%Since $X$ is affiliated to $M$, $X^*$ is affiliated to $M$ too. As $\xi_0\in D(X^*)$, for any $z'\in M'$ we have $z'\xi_0\in D(X^*)$ and $X^*z'\xi_0=z'X^*\xi_0$ so that
%\[
%(z'\xi_0|\eta)=\lim_n(z'\xi_0|x_ny\xi_0)=\lim_n (z'x_n^*\xi_0|y\xi_0)=(z'X^*\xi_0|y\xi_0)=(X^*z'\xi_0|y\xi_0)
%%=(z'\xi_0|Xy\xi_0)
%\]
%and $|(X^*z'\xi_0|y\xi_0)|\le \|z'\xi_0\|\cdot\|\eta\|$. By the density of $M'\xi_0$ in $L^2(M)$ and the closedness of $X=X^{**}$, we finally have $y\xi_0\in D(X)$ and
%$\eta=X(y\xi_0)\in D(S_0)$.\\
As $S_0^2$ is the identity operator on $D(S_0)$, from $S_0(X\xi_0)=X^*\xi_0$ it follows that $X^*\xi_0\in D(S_0)$ and $S_0(X^*\xi_0)=X\xi_0$. Thus $(X^*,D(X^*))$ satisfies the same hypotheses as $(X,D(X))$ and the statements involving $(X^*,D(X^*))$ can be deduced from those involving $(X,D(X))$ proved above, by substitution and the fact that the sub-algebra $M_0$ is involutive.
\vskip0.2truecm\noindent
The operator $L_\xi$ is well defined since $\xi\in D(S_0)=D(\Delta_0^{1/2})$ implies $\Delta_0^{1/4}\xi\in D(\Delta_0^{1/4})$ and $J\Delta_0^{1/4}\xi\in JD(\Delta_0^{1/4})=D(\Delta_0^{-1/4})=D(\Delta_0^{1/4}y\Delta_0^{-1/4})$. Since $\xi\in D(S_0)$ implies $S_0\xi\in D(S_0)$, analogous relations imply that $J\Delta_0^{1/4}S_0\xi\in D(\Delta_0^{1/4}y\Delta_0^{-1/4})$ so that $R_\xi$ is well defined too.
\vskip0.2truecm\noindent
{\bf As first step to prove iii)},  we show that $D(\Delta_0^{1/4})$ is a left $M_0$-module, i.e. $y\zeta\in D(\Delta_0^{1/4})$ for any $y\in M_0$ and $\zeta\in D(\Delta_0^{1/4})$ (a fact probably known in literature). Notice first that since $\sigma^{\o_0}_{-i/4}(y)\in M_0$ we have
\[
\sigma^{\o_0}_{-i/4}(y)\xi_0
%=(\Delta_0^{1/4}y\Delta_0^{-1/4})\xi_0=\Delta_0^{1/4}y\xi_0
\in D(\Delta_0^{1/4})=D(\Delta_0^{-1/4}y^*\Delta_0^{1/4}),
\]
which means, in particular, that $y^*\Delta_0^{1/4}\sigma^{\o_0}_{-i/4}(y)\xi_0\in D(\Delta_0^{-1/4})$ and implies
\[
\begin{split}
|\sigma^{\o_0}_{-i/4}(y)|^2\xi_0&=(\sigma^{\o_0}_{-i/4}(y))^*\sigma^{\o_0}_{-i/4}(y)\xi_0=\sigma^{\o_0}_{i/4}(y^*)\sigma^{\o_0}_{-i/4}(y)\xi_0\\
&=\overline{\Delta_0^{-1/4}y^*\Delta_0^{1/4}}\sigma^{\o_0}_{-i/4}(y)\xi_0=\Delta_0^{-1/4}y^*\Delta_0^{1/4}\sigma^{\o_0}_{-i/4}(y)\xi_0\\
&=\Delta_0^{-1/4}y^*\Delta_0^{1/4}\sigma^{\o_0}_{-i/4}(y)\Delta_0^{1/4}\xi_0\\
&=\Delta_0^{-1/4}y^*\sigma^{\o_0}_{-i/4}(\sigma^{\o_0}_{-i/4}(y))\xi_0\\
&=\Delta_0^{-1/4}y^*\sigma^{\o_0}_{-i/2}(y)\xi_0\\
&=\Delta_0^{-1/4}y^*\Delta_0^{1/2}y\xi_0\\
&=\Delta_0^{-1/4}y^*Jy^*J\xi_0.\\
\end{split}
\]
Since $i_0$ is positivity preserving, setting $c:=\|\sigma^{\o_0}_{-i/4}(y)\|^2$, we thus obtain the bound
\[
y^*Jy^*J\xi_0=\Delta_0^{1/4}(|\sigma^{\o_0}_{-i/4}(y)|^2\xi_0)=i_0(|\sigma^{\o_0}_{-i/4}(y)|^2)\le c\cdot \xi_0.
\]
Consider now a sequence $x_n\xi_0\in M\xi_0$ converging to $\zeta\in D(\Delta_0^{1/4})$ in the graph norm of $(\Delta_0^{1/4},D(\Delta_0^{1/4}))$. Then $\lim_n\|yx_n\xi_0-y\zeta\|\le \|y\|\cdot\lim_n\|x_n\xi_0-\zeta\|=0$. Since
$\Delta_0^{^1/4}x_n\xi_0\in L^2(M)$ is a Cauchy sequence, we have $((x_n-x_m)J(x_n-x_m)J\xi_0|\xi_0)=(x_n\xi_0-x_m\xi_0|J(x_n^*-x_m^*)J\xi_0)=(x_n\xi_0-x_m\xi_0|\Delta_0^{1/2}(x_n-x_m)\xi_0)=\|\Delta_0^{1/4}x_n\xi_0-\Delta_0^{1/4}x_m\xi_0\|^2\to 0$ and, by analogous identities, the self-polarity of $L^2_+(M)$ and the bound above, we get $\|\Delta_0^{1/4}yx_n\xi_0-\Delta_0^{1/4}yx_m\xi_0\|^2=((x_n-x_m)J(x_n-x_m)J\xi_0|y^*Jy^*J\xi_0)\le c\cdot ((x_n-x_m)J(x_n-x_m)J\xi_0|\xi_0)\to 0$. Thus $yx_n\xi_0\in D(\Delta_0^{1/4})$ converges in the graph norm of $(\Delta_0^{1/4},D(\Delta_0^{1/4}))$ to $y\zeta\in D(\Delta_0^{1/4})$. The arbitrariness of $y\in M_0$ and $\zeta \in D(\Delta_0^{1/4})$ implies that $D(\Delta_0^{1/4})$ is an $M_0$-module.
\vskip0.2truecm\noindent
Coming back to the proof of iii), notice that, by the identity $J\Delta_0^{1/4}=\Delta_0^{-1/4}J$, one has $J\Delta_0^{1/4}=(J\Delta_0^{1/4})^{-1}$ so that the closed operator $J\Delta_0^{1/4}$ is idempotent on its domain:
\[
(J\Delta_0^{1/4})\xi\in D(\Delta_0^{1/4}),\qquad(J\Delta_0^{1/4})^2\zeta=\zeta,\qquad \forall\,\zeta\in D(\Delta_0^{1/4}).
\]
Thus, for $y\in M_0$ and $\zeta\in D(\Delta_0^{1/4})$, we have $y^*\zeta=y^*(J\Delta_0^{1/4})^2\zeta$ and, since
$y^*\in M_0$ implies $y^*\zeta\in D(\Delta_0^{1/4})$, we have $y^*\zeta=(J\Delta_0^{1/4})^2y^*(J\Delta_0^{1/4})^2\zeta$ too. Applying this identity to $\zeta:=J\Delta_0^{1/4}\eta$ for any $\eta\in D(\Delta_0^{1/4})$, we have $(J\Delta_0^{1/4})^{-1}y^*(J\Delta_0^{1/4})\eta=(J\Delta_0^{1/4})y^*(J\Delta_0^{1/4})^{-1}\eta$, i.e.
\[
\Delta_0^{-1/4}Jy^*J\Delta_0^{1/4}\eta=J\Delta_0^{1/4}y^*\Delta_0^{-1/4}J\eta\qquad \eta\in D(\Delta_0^{1/4}).
\]
Since, by hypotheses, $X\xi_0\in D(\Delta_0^{1/2})$, we may apply the identity to $\eta:=\Delta_0^{1/4}X\xi_0\in D(\Delta_0^{1/4})$, to get
\[
\begin{split}
\Delta_0^{-1/4}Jy^*J\Delta_0^{1/4}(\Delta_0^{1/4}X\xi_0)&=J\Delta_0^{1/4}y^*\Delta_0^{-1/4}J\Delta_0^{1/4}X\xi_0 \\
&=J\sigma^\o_{-i/4}(y^*)J\Delta_0^{1/4}\xi=L_\xi i_0(y).
\end{split}
\]
Since, by ii), $Xy\xi_0\in D(S_0)=D(\Delta_0^{1/2})\subseteq D(\Delta_0^{1/4})$ and $S_0^{-1}=\Delta_0^{-1/2}J$, we have
\[
Xy\xi_0=S_0^{-1}(S_0(Xy\xi_0))=S_0^{-1}(y^*X^*\xi_0)=\Delta_0^{-1/2}Jy^*J\Delta_0^{1/2}X\xi_0
\]
and we conclude the proof of iii) by
\[
\Delta_0^{1/4}Xy\xi_0=\Delta_0^{-1/4}Jy^*J\Delta_0^{1/4}(\Delta_0^{1/4}X\xi_0)=L_\xi i_0(y).
\]
{\bf To prove iv)}, notice first that, since $S_0^2$ is the identity operator on $D(S_0)$ and $X\xi_0\in D(S_0)$, we have $X^*\xi_0=S_0(X\xi_0)\in D(S_0)=D(\Delta_0^{1/2})$, $X\xi_0=S_0(X^*\xi_0)$, $\Delta_0^{1/4}X^*\xi_0\in D(\Delta_0^{1/4})$ and, for all $y\in M$, $J\Delta_0^{1/4}X^*\xi_0\in D(\Delta_0^{-1/4})=D(\Delta_0^{1/4}y\Delta_0^{-1/4})$. Since, as shown above, $yX\xi_0\in D(\Delta_0^{1/4})$ for all $y\in M_0$ as $X\xi_0\in D(\Delta_0^{1/2})\subset D(\Delta_0^{1/4})$ and $J\Delta_0^{1/2}=\Delta_0^{-1/4}J\Delta_0^{1/4}$ as closed operators, we have
\[
\begin{split}
\Delta_0^{1/4}yX\xi_0&=\Delta_0^{1/4}yS_0(X^*\xi_0)=\Delta_0^{1/4}yJ\Delta_0^{1/2}X^*\xi_0=\Delta_0^{1/4}y\Delta_0^{-1/4}J\Delta_0^{1/4}S_0(\xi)\\
&=\sigma^{\o_0}_{-i/4}(y)J\Delta_0^{1/4}S_0(\xi)=R_\xi i_0(y).
\end{split}
\]
{\bf To prove v)}, i.e. that $\overline{L_\xi}$ is affiliated with $M$, let us start to notice that for $z'\in M_0'$ and $y\in M_0$, setting $z:=J\Delta_0^{-1/4}z'^*\Delta_0^{1/4}J=\Delta_0^{1/4}Jz'^*J\Delta_0^{-1/4}=\sigma^{\o_0}_{-1/4}(Jz'^*J)\in M$, since $\Delta_0^{-1/4}J=J\Delta_0^{1/4}$ and $\Delta_0^{-1/2}z'\xi_0\in D(\Delta_0^{1/2})\subset D(\Delta_0^{1/4})=D(J\Delta_0^{1/4})$, we have $z\xi_0=J\Delta_0^{-1/4}z'^*\Delta_0^{1/4}J\xi_0=J\Delta_0^{-1/4}z'^*\xi_0=J\Delta_0^{-1/4}J\Delta_0^{-1/2}z'\xi_0=\\=JJ\Delta_0^{1/4}\Delta_0^{-1/2}z'\xi_0=\Delta_0^{-1/4}z'\xi_0$, $z'\xi_0=\Delta_0^{1/4}z\xi_0=\sigma^{\o_0}_{-i/4}(z)\xi_0$ and
\[
\begin{split}
z'i_0(y)&=z'\Delta_0^{1/4}y\Delta_0^{-1/4}\xi_0=z'\sigma^{\o_0}_{-i/4}(y)\xi_0=\sigma^{\o_0}_{-i/4}(y)z'\xi_0 \\
&=\sigma^{\o_0}_{-i/4}(y)\sigma^{\o_0}_{-i/4}(z)\xi_0=\sigma^{\o_0}_{-i/4}(yz)\xi_0=i_0(yz)
\end{split}
\]
so that $z'i_0(y)\in M_0\xi_0=D(L_\xi)$. Since $\sigma^{\o_0}_{i/4}(z')\in M'$, $y\in M$, $X$ is affiliated to $M$, $y\xi_0\in D(X)$, $yz\xi_0\in D(X)$ by i), using iii) we have
\[
\begin{split}
L_\xi z'i_0(y)=L_\xi i_0(yz)&=\Delta_0^{1/4}Xyz\xi_0\\
&=\Delta_0^{1/4}Xy\Delta_0^{-1/4}z'\xi_0\\
&=\Delta_0^{1/4}Xy\Delta_0^{-1/4}z'\Delta_0^{1/4}\xi_0\\
&=\Delta_0^{1/4}Xy\sigma^{\o_0}_{i/4}(z')\xi_0\\
&=\Delta_0^{1/4}\sigma^{\o_0}_{i/4}(z')Xy\xi_0\\
&=\Delta_0^{1/4}\sigma^{\o_0}_{i/4}(z')\Delta_0^{-1/4}\Delta_0^{1/4}Xy\xi_0\\
&=\sigma^{\o_0}_{-i/4}(\sigma^{\o_0}_{i/4}(z'))L_\xi i_0(y)\\
&=z'L_\xi i_0(y).
\end{split}
\]
Since by Lemma 2.7 $\RR_\eta=J\L_\eta J$, for $\eta\in D(S_0)$ we have  $R_\xi=JL_\xi J$ for $\xi\in D(\Delta_0^{1/2})$. This is equivalent to $(J\oplus J)\mathcal{G}(R_\xi)=\mathcal{G}(L_\xi)$ and implies $(J\oplus J)\mathcal{G}(\overline{R_\xi})=\mathcal{G}(\overline{L_\xi})$, i.e. $\overline{R_\xi}=J\overline{L_\xi} J$ as an identity between densely defined closed operators.\\
To prove vi), notice that, by i) we have $i_0(y)\in D(\Delta_0^{-1/4})$ and
$\Delta_0^{-1/4}i_0(y)=y\xi_0\in D(X)$ and by ii) we have that $X\Delta_0^{-1/4}i_0(y)=Xy\xi_0\in D(S_0)\subseteq D(\Delta_0^{1/4})$.  Hence $i_0(y)\in D(\Delta_0^{1/4}X\Delta_0^{-1/4})$ and $(\Delta_0^{1/4}X\Delta_0^{-1/4})i_0(y)=\\
\Delta_0^{1/4}Xy\xi_0=L_\xi i_0(y)$.\\
To prove vii), notice that, since $y\xi_0\in D(S_0)$ and $S_0=\Delta_0^{-1/4}J\Delta_0^{1/4}$ on $D(S_0)$, we have $y^*\xi_0=S_0(y\xi_0)=\Delta_0^{-1/4}J\Delta_0^{1/4}y\xi_0=\Delta_0^{-1/4}Ji_0(y)$.
By i) and ii), $y^*\xi_0\in D(X^*)$, $X^*y^*\xi_0\in D(S_0)$ and $yX\xi_0=S_0(X^*y^*\xi_0)\in D(S_0)\subset D(\Delta_0^{1/4})$ so that by iv), $R_\xi i_0(y)=\Delta_0^{1/4}yX\xi_0=(\Delta_0^{1/4}S_0)(X^*y^*\xi_0)$. Since $\Delta_0^{1/4}S_0=J\Delta_0^{1/4}$ on $D(S_0)$, we have
\[
R_\xi i_0(y)=(\Delta_0^{1/4}S_0)(X^*y^*\xi_0)=(J\Delta_0^{1/4})(X^*y^*\xi_0)=(J\Delta_0^{1/4}X^*\Delta_0^{-1/4}J)i_0(y)
\]
showing that $J\Delta_0^{1/4}X^*\Delta_0^{-1/4}J$ is densely defined on $i_0(M_0)$ and there it coincides with $R_\xi$.
\vskip0.2truecm\noindent
To prove the first identity in viii), notice that, by the Spectral Theorem, $\xi=X\xi_0$ is an eigenvector of $\Delta_0^{1/4}$ with eigenvalue $\lambda>0$: $\Delta_0^{1/4}X\xi_0=\lambda\cdot X\xi_0$. By the density of $M_0'\xi_0$ in $L^2(M)$ and for all $z'\in M_0'$ we then have
\[
\begin{split}
(z'\xi_0|L_\xi i_0(y))&=(z'\xi_0|\Delta_0^{1/4}Xy\xi_0)\\
&=(\Delta_0^{1/4}z'\Delta_0^{-1/4}\xi_0|Xy\xi_0)\\
&=(\Delta_0^{1/4}z'\Delta_0^{-1/4}X^*\xi_0|y\xi_0)\\
&=(\Delta_0^{1/4}z'\Delta_0^{-1/4}J\Delta_0^{1/2}X\xi_0|y\xi_0)\\
&=\lambda^2\cdot(\Delta_0^{1/4}z'\Delta_0^{-1/4}JX\xi_0|y\xi_0)\\
&=\lambda^2\cdot(\Delta_0^{1/4}z'J\Delta_0^{1/4}X\xi_0|y\xi_0)\\
&=\lambda^2\cdot(z'J\Delta_0^{1/4}X\xi_0|\Delta_0^{1/4}y\xi_0)\\
&=\lambda^2\cdot(z'J\Delta_0^{1/4}X\xi_0|i_0(y))\\
&=\lambda^{3}\cdot(z'JX\xi_0|i_0(y))\\
&=\lambda\cdot(z'X^*\xi_0|i_0(y))\\
&=(z'\xi_0|\lambda Xi_0(y)).
\end{split}
\]
To prove the second identity in viii), we first need to show that the adjoint of the densely defined operator $(\Delta_0^{1/4}X\Delta_0^{-1/4},i_0(M_0))$ (which is closable by ii) and vi)) is an extension of the well defined and densely defined operator  $(\Delta_0^{-1/4}X^*\Delta_0^{1/4},M_0\xi_0)$. By ii), for $x\in M_0$ and since
$i_0(x)=\sigma^{\o_0}_{-i/4}(x)\xi_0\in M_0\xi_0\subset D(X^*)$, we have $JX^*i_0(x)=JX^*JJi_0(x)=JX^*Ji_0(x^*)=\\
JX^*J\sigma^{\o_0}_{-i/4}(x^*)\xi_0$. Since $JX^*J$ is affiliated to $M'$, $\sigma^{\o_0}_{-i/4}(x^*)\in M$ and by ii) $\sigma^{\o_0}_{-1/4}(x^*)\xi_0\in D(JX^*J)$, we have
\[
\begin{split}
JX^*i_0(x)&=\sigma^{\o_0}_{-i/4}(x^*)JX^*J\xi_0=\sigma^{\o_0}_{-i/4}(x^*)JX^*\xi_0
=\sigma^{\o_0}_{-i/4}(x^*)JS_0(X\xi_0) \\
&=\sigma^{\o_0}_{-i/4}(x^*)\Delta_0^{1/2}(X\xi_0).
\end{split}
\]
The hypothesis that $X\xi_0$ is an eigenvalue of $\Delta_0^{1/2}$ then implies $JX^*i_0(x)=\lambda^2\cdot \sigma^{\o_0}_{-i/4}(x^*)X\xi_0$ which in turn, by ii), implies $JX^*i_0(x)=
\lambda^2\cdot S_0(X^*\sigma^{\o_0}_{i/4}(x)\xi_0)\in D(S_0)=D(\Delta_0^{1/2})\subset D(\Delta_0^{1/4})$ so that $X^*\Delta_0^{1/4}x\xi_0=X^*i_0(x)\in JD(\Delta_0^{1/4})=D(\Delta_0^{-1/4})$ and $x\xi_0\in D(\Delta_0^{-1/4}X^*\Delta_0^{1/4})$. For all $y\in M_0$ we may then compute
\[
\begin{split}
(x\xi_0|(\Delta_0^{1/4}X\Delta_0^{-1/4})i_0(y))&=(\Delta_0^{1/4}x\xi_0|(X\Delta_0^{-1/4})i_0(y)) \\
&=(X^*\Delta_0^{1/4}x\xi_0|\Delta_0^{-1/4}i_0(y)) \\
&=((\Delta_0^{-1/4}X^*\Delta_0^{1/4})x\xi_0|i_0(y))
\end{split}
\]
so that $x\xi_0\in D((\Delta_0^{1/4}X\Delta_0^{-1/4})^*)$ and $(\Delta_0^{-1/4}X^*\Delta_0^{1/4})x\xi_0=\\
(\Delta_0^{1/4}X\Delta_0^{-1/4})^*x\xi_0$. Since by vi) and the first identity proved above we have
\[
(\Delta_0^{1/4}X\Delta_0^{-1/4})x\xi_0=L_\xi x\xi_0=\lambda\cdot Xx\xi_0
\]
for all $x\xi_0\in M_0\xi_0=i_0(M_0)=D(L_\xi)\subset D(\Delta_0^{1/4}X\Delta_0^{-1/4})\cap D(X)$, by i) we then have
\[
(\Delta_0^{-1/4}X^*\Delta_0^{1/4})x\xi_0=\lambda X^* x\xi_0\qquad x\xi_0\in M_0\xi_0\subset D(\Delta_0^{-1/4}X^*\Delta_0^{1/4})\cap D(X^*).
\]
{\bf To finalize the proof of the second identity in viii)}, rewrite the eigenvalue equation satisfied by $\xi=X\xi_0$ as $JX^*J\xi_0=\Delta_0^{1/2}X\xi_0=\lambda^2 X\xi_0$ so that, for all $y\in M_0$, we have
\[
\begin{split}
R_\xi i_0(y)&=\Delta_0^{1/4}yX\xi_0\\
&=\lambda^{-2}\cdot \Delta_0^{1/4}y(JX^*J)\xi_0\\
&=\lambda^{-2}\cdot \Delta_0^{1/4}(JX^*J)y\xi_0\\
&=\lambda^{-2}\cdot J(\Delta_0^{-1/4}X^*\Delta_0^{1/4})J\Delta_0^{1/4}y\xi_0\\
&=\lambda^{-2}\cdot J(\Delta_0^{-1/4}X^*\Delta_0^{1/4})Ji_0(y)\\
&=\lambda^{-2}\cdot J(\Delta_0^{-1/4}X^*\Delta_0^{1/4})i_0(y^*)\\
&=\lambda^{-2}\cdot J(\Delta_0^{-1/4}X^*\Delta_0^{1/4})\sigma^{\o_0}_{-i/4}(y^*)\xi_0\\
&=\lambda^{-2}\cdot JX^*\sigma^{\o_0}_{-i/4}(y^*)\xi_0\\
&=\lambda^{-2}\cdot JX^*Ji_0(y).
\end{split}
\]
\end{proof}

\begin{lem}
Let $\xi\in D(S_0)$ be eigenvector of $\Delta_0^{1/2}$ corresponding to the eigenvalue $\lambda^2>0$ and fix a densely defined, closed operator $(X,D(X))$ affiliated to $M$ such that
\[
\xi_0\in D(X)\cap D(X^*),\qquad \xi=X\xi_0,\quad S_0(X\xi_0)=X^*\xi_0.
\]
Then $S_0\xi\in D(S_0)=D(\Delta_0^{1/2})$ is eigenvector of $\Delta_0^{1/2}$ corresponding to the eigenvalue $\lambda^{-2}$.
\end{lem}
\begin{proof}
On one hand we have $S_0\xi=J\Delta_0^{1/2}\xi=\lambda^2\cdot J\xi$. On the other hand, since $JD(\Delta_0^{1/2})=D(\Delta_0^{-1/2})$ and $S_0=S_0^{-1}$ on $D(\Delta_0^{1/2})=D(S_0)=D(S_0^{-1})$ we have $S_0\xi=S_0^{-1}\xi=\Delta_0^{-1/2}J\xi$ so that $\Delta_0^{1/2}S_0\xi=J\xi=\lambda^{-2}\cdot S_0\xi$.
\end{proof}
Combining the results obtained, we have
\begin{cor}
Let $\xi\in D(S_0)$ be an eigenvector of $\Delta_0^{1/2}$ corresponding to the eigenvalue $\lambda^2>0$ and fix a densely defined, closed operator $(X,D(X))$ affiliated to $M$ such that
\[
\xi_0\in D(X)\cap D(X^*),\qquad \xi=X\xi_0,\quad S_0(X\xi_0)=X^*\xi_0.
\]
Then for all $y\in M_0$ we have
\begin{equation}
(L_\xi-R_\xi)i_0(y)=\Delta_0^{1/4}(Xy-yX)\xi_0\\
\end{equation}
and
\begin{equation}
(L_{S_0\xi}-R_{S_0\xi})i_0(y)=\Delta_0^{1/4}(X^*y-yX^*)\xi_0.
\end{equation}
\end{cor}
The commutator $[X,y]:=Xy-yX$ is in general only densely defined if $X$
is affiliated to $M$ but, within the hypotheses assumed at the beginning of this section, the vector $\xi_0$ belongs to the domain of $[X,y]$ and its image $[X,y]\xi_0$ belongs to $D(\Delta_0^{1/4})$. This may justify the notation
\[
i_0([X,y]):=\Delta_0^{1/4}(Xy-yX)\xi_0\qquad y\in M_0.
\]
In the following we will use the notation $j(X^*):=JX^*J$.
\vskip0.1truecm\noindent
Next result shows that the symmetric embedding $i_0$ intertwines the unbounded spatial derivations $\delta_X$, $\delta_{X^*}$ on $M$ with the unbounded bimodule derivations $d^\lambda_X$, $d^{\lambda^{-1}}_{X^*}$ on $L^2(M)$.

\begin{prop} (Bimodule derivations and spatial derivations)
Let $\xi\in D(S_0)$ be eigenvector of $\Delta_0^{1/2}$ corresponding to the eigenvalue $\lambda^2>0$ and fix a densely defined, closed operator $(X,D(X))$ affiliated to $M$ such that
\[
\xi_0\in D(X)\cap D(X^*),\qquad \xi=X\xi_0,\quad S_0(X\xi_0)=X^*\xi_0.
\]
Then, setting $d_X^{\lambda}:=i(\lambda X-\lambda^{-1} j(X^*))$ and $d_{X^*}^{\lambda^{-1}} :=i(\lambda^{-1} X^*-\lambda j(X))$, we have
\[
d_X^{\lambda} i_0(y)= i_0(i[X,y])\qquad d_{X^*}^{\lambda^{-1} } i_0(y)=
i_0(i[X^*,y])\qquad y\in M_0.
\]
Otherwise stated, setting $\delta_X(y):=i[X,y]$ for any $y\in M_0$, on the $^*$-algebra $M_0$ we have
\[
d_X^{\lambda}\circ i_0= i_0\circ \delta_X\qquad d_{X^*}^{\lambda^{-1}}\circ i_0= i_0\circ \delta_{X^*}.
\]
\end{prop}

\begin{proof}
For $y\in M_0$ we have $\sigma^{\o_0}_{-i/4}(y)\in M_0$, $J\sigma^{\o_0}_{-i/4}(y)J\in M_0'$ and
\[
J\sigma^{\o_0}_{-i/4}(y)J\xi_0=J\sigma^{\o_0}_{-i/4}(y)\xi_0=J\Delta_0^{1/4}y\xi_0=Ji_0(y).
\]
Since $X^*$ is affiliated with $M$ and $\xi_0\in D(X^*)$, we have $(J\sigma^{\o_0}_{-i/4}(y)J)\xi_0\in D(X^*)$ and
\[
\begin{split}
j(X^*)i_0(y)&=JX^*Ji_0(y) \\
&=JX^*(J\sigma^{\o_0}_{-i/4}(y)J)\xi_0\\
&=J(J\sigma^{\o_0}_{-i/4}(y)J)X^*\xi_0 \\
&=\sigma^{\o_0}_{-i/4}(y)JX^*\xi_0 \\
&=\sigma^{\o_0}_{-i/4}(y)\Delta_0^{1/2}X\xi_0 \\
&=\Delta_0^{1/4}y\Delta_0^{1/4}X\xi_0 \\
&=\lambda\cdot\Delta_0^{1/4}yX\xi_0.
\end{split}
\]
Since $\lambda Xi_0(y)=L_\xi i_0(y)=\Delta_0^{1/4}Xy\xi_0$ we have too $d_X^\lambda i_0(y):=i(\lambda X-\lambda^{-1} j(X^*))i_0(y)=i\Delta_0^{1/4}(Xy\xi_0-yX\xi_0)=i_0(i[X,y])=i_0(\delta_X(y))$.
The proof of the second identity is similar.
\end{proof}

\begin{thm}
Let $\xi\in D(S_0)$ be an eigenvector of $\Delta_0^{1/2}$ corresponding to the eigenvalue $\lambda^2>0$ and $(X,D(X))$ a densely defined, closed operator affiliated to $M$ such that
\[
\xi_0\in D(X)\cap D(X^*),\qquad \xi=X\xi_0,\quad S_0(\xi)=X^*\xi_0.
\]
Then the completely Dirichlet form $(\E_X^\lambda, \F_X^\lambda)$ constructed above may be represented as
\[
\begin{split}
\E_X^\lambda[i_0(y)]&=\|i_0([X,y])\|^2_{L^2(M)} + \|i_0([X^*,y])\|^2_{L^2(M)}\qquad y\in M_0\\
&=\|i_0(\delta_X(y))\|^2_{L^2(M)} + \|i_0(\delta_{X^*}(y))\|^2_{L^2(M)}
\end{split}
\]
on the $L^2(M)$-dense, $J$-invariant subspace $M_0\xi_0=i_0(M)\subset \tilde \F_X\subset \F^\lambda_X$.
\end{thm}
\begin{rem}
These results prove a fortiori that and under the stated assumptions, the form
\[
i_0(y)\mapsto   \lambda^2\|i_0([X,y])\|^2_{L^2(M)} + \|i_0([X^*,y])\|^2_{L^2(M)}
\]
extends to a completely Dirichlet form on $L^2(M)$ with respect to the cyclic vector $\xi_0\in L^2_+(M)$. If $\xi_0$ would be the vector representing a finite, normal, faithful trace state $\o_0$, this result would follow from the general theory relating completely Dirichlet forms and closable bimodule derivations on von Neumann algebras with trace (see \cite{cipsau}).
\end{rem}

\section{Coercivity of Dirichlet forms}

In this section we still keep the assumption that $\xi\in D(S_0)$ is an eigenvector of $\Delta_0^{1/2}$ corresponding to the eigenvalue $\lambda^2>0$ (we still assume $\lambda>0$) and $(X,D(X))$ a densely defined, closed operator affiliated to $M$ such that
\[
\xi_0\in D(X)\cap D(X^*),\qquad \xi=X\xi_0,\quad S_0(\xi)=X^*\xi_0.
\]
We prove below natural lower bounds on the Dirichlet form $(\E_X^\lambda, \F_X^\lambda)$ constructed in Section 2, which lead to coercivity. Recall that $(\E_X^\lambda, \F_X^\lambda)$ is defined as the closure of the
densely defined, $J$-real, closable quadratic form $\tilde\E_X^\lambda:\tilde\F_X\to [0,+\infty)$ on $L^2(M)$ given by
\[
\tilde\E_X^\lambda[\eta]:=\|d_X^\lambda\eta\|^2_{L^2(M)} + \|d_{X^*}^{\lambda^{-1}}\eta\|^2_{L^2(M)}\qquad
\eta\in\tilde\F_X=D(d_X^\lambda)\cap D(d_{X^*}^{\lambda^{-1}}),
\]
where $d_X^{\lambda}:=i(\lambda X-\lambda^{-1} j(X^*))$ and $d_{X^*}^{\lambda^{-1}} :=i(\lambda^{-1} X^*-\lambda j(X))$ are defined on the domain
\[
\tilde\F_X=D(X)\cap D(X^*)\cap J(D(X)\cap D(X^*))
\]
containing the $L^2(M)$-dense, $J$-invariant subspace $i_0(M_0)=M_0\xi_0\subset \tilde\F_X$. Obviously $\tilde\F_X$ is a form core for $(\E^\lambda_X,\F^\lambda_X)$ and on it $\E_X^{\lambda}$ and $\tilde\E_X^\lambda$ coincide.
\vskip0.2truecm\noindent
We start showing an alternative representation of the Dirichlet form.
\begin{thm}
The following representation holds true for the quadratic form $(\E_X^\lambda, \tilde\F_X)$:
\begin{equation}
\begin{split}
\E_X^\lambda[\eta]&=\lambda^2(\|X\eta\|^2+\|XJ\eta\|^2)+\lambda^{-2}(\|X^*\eta\|^2+\|X^*J\eta\|^2)\\
&\quad-2\bigl[(X\eta |JX^*J\eta)+(X^*\eta|JXJ\eta)\bigr]\qquad\qquad\qquad\qquad \eta\in \tilde\F_X.
\end{split}
\end{equation}
\end{thm}
\begin{proof}
In the following, we repeatedly use the fact that if $N\subseteq B(h)$ is a von Neumann algebra acting on a Hilbert space $h$ and $(A,D(A))$, $(B,D(B))$ are densely defined, closed operator on $h$ affiliated to $N$ and $N'$, respectively, then
\[
(A\eta|B\zeta)=(B^*\eta|A^*\zeta)\qquad \eta\in D(A)\cap D(B^*),\quad \zeta\in D(B)\cap D(A^*).
\]
This identity follows directly if $B\in M'$ is bounded since then $B^*\in M'$ and $\eta\in D(A)$ implies $B^*\eta\in D(A)$ and $AB^*\eta=B^*A\eta$ so that $(A\eta|B\zeta)=(B^*A\eta|\zeta)=(AB^*\eta|\zeta)=(B^*\eta|A^*\zeta)$. In general we may approximate $B$ weakly by $B_\varepsilon :=B(I+\varepsilon |B|)^{-1}\in M'$ as $\varepsilon\downarrow 0$.\\
We start the proof of the result setting
\[
d_X:=i(X-j(X^*)),\qquad V_X^\lambda:=i(1-\lambda^{-1})(\lambda X+j(X^*))
\]
and using the splittings
\[
d_X^\lambda=d_X+V_X^\lambda\qquad d_{X^*}^{\lambda^{-1}}=d_{X^*}+V_{X^*}^{\lambda^{-1}},
\]
for any $\eta\in \tilde\F_X$ to have the representation
\begin{equation}
\begin{split}
\E_X^\lambda[\eta]:=&\|d_X\eta\|^2 + \|d_{X^*}\eta\|^2 + \|V_X^\lambda\eta\|^2 + \|V_{X^*}^{\lambda^{-1}}\eta\|^2 +\\
&(d_X\eta| V_X^\lambda\eta) + (V_X^\lambda\eta|d_X\eta) + (d_{X^*}\eta| V_{X^*}^{\lambda^{-1}}\eta) + (V_{X^*}^{\lambda^{-1}}\eta|d_{X^*}\eta).
\end{split}
\end{equation}
Since
\[
\begin{split}
\|d_X\eta\|^2=&\|(X-j(X^*))\eta\|^2=\|X\eta\|^2+\|j(X^*)\eta\|^2 \\
&-[(X\eta|j(X^*)\eta)+(X^*\eta| j(X)\eta)],\\
\|d_{X^*}\eta\|^2=&\|(X^*-j(X))\eta\|^2=\|X^*\eta\|^2+\|j(X)\eta\|^2 \\
&-[(X\eta|j(X^*)\eta)+(X^*\eta| j(X)\eta)]
\end{split}
\]
the sum of the first two addends in (3.2) equals
\begin{equation}
\begin{split}
\|d_X\eta\|^2+\|d_{X^*}\eta\|^2=&\|X\eta\|^2+\|j(X^*)\eta\|^2+\|X^*\eta\|^2+\|j(X)\eta\|^2\\
&-2[(X\eta|j(X^*)\eta)+(X^*\eta| j(X)\eta)].
\end{split}
\end{equation}
Since also
\[
\begin{split}
 \|V_X^\lambda\eta\|^2&=(1-\lambda^{-1})^2((\lambda X+j(X^*))\eta|(\lambda X+j(X^*))\eta)\\
=&(1-\lambda^{-1})^2\bigl[\lambda^2 \|X\eta\|^2+\|j(X^*)\eta\|^2 +\\
&\lambda ((X\eta|j(X^*)\eta)+(X^*\eta|j(X)\eta)\bigr],\\
\|V_{X^*}^{\lambda^{-1}}\eta\|^2&=(1-\lambda)^2((\lambda^{-1} X^*+j(X))\eta|(\lambda^{-1} X^*+j(X))\eta)\\
=&(1-\lambda)^2\bigl[\lambda^{-2} \|X^*\eta\|^2+\|j(X)\eta\|^2+ \\
&\lambda^{-1} ((X^*\eta|j(X)\eta)+(X\eta|j(X^*)\eta)\bigr]
\end{split}
\]
the sum of the third and fourth addends in (3.2) equals
\begin{equation}
\begin{split}
&\|V_X^\lambda\eta\|^2+\|V_{X^*}^{\lambda^{-1}}\eta\|^2= \\
&(\lambda-1)^2(\|X\eta\|^2+\|j(X)\eta\|^2)+(\lambda^{-1}-1)^2(\|X^*\eta\|^2+\|j(X^*)\eta\|^2)+\\
&[(1-\lambda^{-1})^2\lambda +(1-\lambda)^2\lambda^{-1}]((X\eta|j(X^*)\eta)+(X^*\eta|j(X)\eta))=\\
&(\lambda-1)^2(\|X\eta\|^2+\|j(X)\eta\|^2)+(\lambda^{-1}-1)^2(\|X^*\eta\|^2+\|j(X^*)\eta\|^2)+\\
&2(\lambda-1)^2\lambda^{-1}\bigl((X\eta|j(X^*)\eta)+(X^*\eta|j(X)\eta)\bigr).
\end{split}
\end{equation}
Since we have too
\[
\begin{split}
&(d_X\eta| V_X^\lambda\eta) + (V_X^\lambda\eta|d_X\eta)= \\
&(1-\lambda^{-1})\bigl[((X-j(X^*))\eta|(\lambda X+j(X^*))\eta)+((\lambda X+j(X^*))\eta|(X-j(X^*))\eta)\bigr]=\\
&(1-\lambda^{-1})\bigl[\lambda \|X\eta\|^2+(X\eta|j(X^*)\eta)-\lambda (X^*\eta|j(X)\eta)-\|j(X^*)\eta\|^2+\\
&\lambda  \|X\eta\|^2-\lambda (X\eta| j(X^*)\eta)+(X^*\eta|j(X)\eta)-\|j(X^*)\eta\|^2\bigr],
\end{split}
\]
the sum of the fifth and sixth addends in (3.2) equals
\begin{equation}
\begin{split}
&(d_X\eta| V_X^\lambda\eta) + (V_X^\lambda\eta|d_X\eta)=\\
&(1-\lambda^{-1})\bigl[2\lambda \|X\eta\|^2-2\|j(X^*)\eta\|^2+\\
&(1-\lambda)((X\eta|j(X^*)\eta) + (X^*\eta|j(X)\eta))\bigr]
\end{split}
\end{equation}
and, analogously, the sum of the seventh and  eighth addends in (3.2) equals
\begin{equation}
\begin{split}
&(d_{X^*}\eta| V_{X^*}^{\lambda^{-1}}\eta) + (V_{X^*}^{\lambda^{-1}}\eta|d_{X^*}\eta)=\\
&(1-\lambda)\bigl[2\lambda^{-1} \|X^*\eta\|^2-2\|j(X)\eta\|^2+\\
&(1-\lambda^{-1})((X^*\eta|j(X)\eta) + (X\eta|j(X^*)\eta))\bigr].
\end{split}
\end{equation}
By substitution of (3.6), (3.5) and (3.4) in (3.2) we obtain
\[\begin{split}
&\E_X^\lambda[\eta]-\bigl(\|d_X\eta\|^2 + \|d_{X^*}\eta\|^2\bigr)=\\
&(\lambda-1)^2(\|X\eta\|^2+\|j(X)\eta\|^2)+(\lambda^{-1}-1)^2(\|X^*\eta\|^2+\|j(X^*)\eta\|^2)+\\
&2(\lambda-1)^2\lambda^{-1}((X\eta|j(X^*)\eta)+(X^*\eta|j(X)\eta)+\\
&(1-\lambda^{-1})\bigl[2\lambda \|X\eta\|^2-2\|j(X^*)\eta\|^2+\\
&(1-\lambda)((X\eta|j(X^*)\eta) + (X^*\eta|j(X)\eta))\bigr]+\\
&(1-\lambda)\bigl[2\lambda^{-1} \|X^*\eta\|^2-2\|j(X)\eta\|^2+\\
&(1-\lambda^{-1})((X^*\eta|j(X)\eta) + (X\eta|j(X^*)\eta))\bigr]\\
=&\bigl[(\lambda-1)^2+2(\lambda-1)\bigr](\|X\eta\|^2+\|j(X)\eta\|^2)+\\
&\bigl[(\lambda^{-1}-1)^2+2(\lambda^{-1}-1)\bigr](\|X^*\eta\|^2+\|j(X^*)\eta\|^2)+\\
&\bigl[2(1-\lambda)(1-\lambda^{-1})+2(\lambda-1)^2\lambda^{-1}\bigr]((X\eta|j(X^*)\eta)+(X^*\eta|j(X)\eta))\\
=&(\lambda^2-1)(\|X\eta\|^2+\|j(X)\eta\|^2)+(\lambda^{-2}-1)(\|X^*\eta\|^2+\|j(X^*)\eta\|^2)+ \\
&\bigl[2(1-\lambda)(\lambda-1)\lambda^{-1}+2(\lambda-1)^2\lambda^{-1}\bigr]((X\eta|j(X^*)\eta)+(X^*\eta|j(X)\eta))\\
=&(\lambda^2-1)(\|X\eta\|^2+\|j(X)\eta\|^2)+(\lambda^{-2}-1)(\|X^*\eta\|^2+\|j(X^*)\eta\|^2)+\\
&\bigl[-2(\lambda-1)^2\lambda^{-1}+2(\lambda-1)^2\lambda^{-1}\bigr]((X\eta|j(X^*)\eta)+(X^*\eta|j(X)\eta))\\
=&(\lambda^2-1)(\|X\eta\|^2+\|j(X)\eta\|^2)+(\lambda^{-2}-1)(\|X^*\eta\|^2+\|j(X^*)\eta\|^2)
\end{split}
\]
and then, by (3.3), we finally obtain (3.1) for any $\eta\in \tilde\F_X$
\[
\begin{split}
\E_X^\lambda[\eta]=&\bigl(\|d_X\eta\|^2 + \|d_{X^*}\eta\|^2\bigr)+(\lambda^2-1)(\|X\eta\|^2+\|j(X)\eta\|^2)+\\
&(\lambda^{-2}-1)(\|X^*\eta\|^2+\|j(X^*)\eta\|^2)\\
=&\lambda^2(\|X\eta\|^2+\|XJ\eta\|^2)+\lambda^{-2}(\|X^*\eta\|^2+\|X^*J\eta\|^2)-\\
&2\bigl[(X\eta |JX^*J\eta)+(X^*\eta|JXJ\eta)\bigr].
\end{split}
\]
\end{proof}
\begin{cor} (Lower bound)
The following lower bounds hold true for any $\varepsilon,\delta>0$ and any $\eta\in \tilde\F_X$
\begin{equation}
\begin{split}
\E_X^\lambda[\eta]\ge &(\lambda^2-\varepsilon^2)\|X\eta\|^2 + (\lambda^2-\delta^{-2})\|XJ\eta\|^2+(\lambda^{-2}-\delta^2)\|X^*\eta\|^2 +\\
& (\lambda^{-2}-\varepsilon^{-2})\|X^*J\eta|\|^2.
\end{split}
\end{equation}
In particular, for $\varepsilon=\delta=1$ and any $\eta\in \tilde \F_X$ we have
\begin{equation}
\E_X^\lambda[\eta]\ge (\lambda^2-1)\bigl(\|X\eta\|^2 + \|XJ\eta\|^2\bigr)+(\lambda^{-2}-1)\bigl(\|X^*\eta\|^2 + \|X^*J\eta|\|^2\bigr).
\end{equation}
\end{cor}
\begin{proof}
The result follows from (3.1) and the identities, valid for $\varepsilon,\delta>0$,
\[
\varepsilon^2\|X\eta\|^2 +\varepsilon^{-2}\|j(X^*)\eta\|^2-[(X\eta|j(X^*)\eta)+(X^*\eta|j(X)\eta)]=\|d_X^\varepsilon\eta\|^2\ge 0
\]
\[
\delta^2\|X^*\eta\|^2 +\delta^{-2}\|j(X)\eta\|^2-[(X^*\eta|j(X)\eta)+(X\eta|j(X^*)\eta)]=\|d_{X^*}^\delta\eta\|^2\ge 0.
\]
\end{proof}
We address now the problem to find conditions on $(X,D(X))$ sufficient to guarantee that the lower bounds above are coercive for our Dirichlet form. By this we mean bounds in which the Dirichlet form dominates a quadratic form with a certain degree of discreteness of the spectrum such as existence and finite degeneracy of a ground state, spectral gaps or emptiness of essential spectrum.
The conditions will be formulated in terms of relative smallness of the quadratic form of the self-commutator $[X,X^*]$ with respect to the quadratic form of $X^*X$ and they will be exploited in Section 5 when $M$ is a type I$_\infty$ factor.\\
Let us denote by $(t_X,D(t_X))$ and $(t_{X^*},D(t_X^*))$, the densely defined, positive, closed quadratic forms defined as
\[
\begin{split}
t_X[\eta]&:=\|X\eta\|^2\qquad \eta\in D(t_X):=D(X),\\
t_{X^*}[\eta]&:=\|X^*\eta\|^2\qquad \eta\in D(t_{X^*}):=D(X^*),
\end{split}
\]
whose associated positive, self-adjoint operators are $(X^*X,D(X^*X))$ and $(XX^*,D(XX^*))$.\\
Consider also the quadratic form $(\tilde q_X^\lambda,D(\tilde q_X^\lambda))$ given by
\[
\tilde q_X^\lambda[\eta]:=(\lambda^2-1)\|X\eta\|^2+(\lambda^{-2}-1)\|X^*\eta\|^2\qquad \eta\in D(\tilde q_X^\lambda):=D(X)\cap D(X^*).
\]
By the densely defined quadratic form $(q_0, D(q_0))$ defined as
\[
q_0[\eta]:=t_{X^*}[\eta]-t_X[\eta]=\|X^*\eta\|^2-\|X\eta\|^2\qquad \eta\in D(q_0):=D(X)\cap D(X^*),
\]
on $D(\tilde q_X^\lambda)=D(X)\cap D(X^*)$ we can write
\[
\tilde q_X^\lambda=(\lambda-\lambda^{-1})^2\cdot t_X+(\lambda^{-2}-1)\cdot q_0=(\lambda-\lambda^{-1})^2\cdot t_{X^*}+(1-\lambda^2)\cdot q_0
\]
and regard $\tilde q_X^\lambda$ as a perturbation of a multiple of $t_X$ or $t_{X^*}$ by a multiple of $q_0$. Notice that $q_0$ is the form of the self-commutator $[X,X^*]=XX^*-X^* X$,
at least on $D(X^*X)\cap D(XX^*)$.\\
Using the quadratic form $(\tilde{\mathcal{Q}}^\lambda_X,\tilde\F_X)$ given by
\[
\tilde{\mathcal{Q}}^\lambda_X[\eta]:=\tilde q_X^\lambda[\eta]+\tilde q_X^\lambda[J\eta]\qquad \eta\in \tilde\F_X=D(X)\cap D(X^*)\cap J(D(X)\cap D(X^*)),
\]
the lower bound (3.8) can be written as
\begin{equation}
\tilde{\mathcal{Q}}^\lambda_X[\eta]\le \tilde\E^\lambda_X[\eta]\qquad \eta\in \tilde\F_X.
\end{equation}
Although $\tilde{\mathcal{Q}}^\lambda_X$ is densely defined, since $i_0(M_0)=M_0\xi_0\subset \tilde\F_X$ by Lemma 2.8 ii), it is not necessarily lower bounded, closable or a proper functional.
\vskip0.1truecm\noindent
For sake of clarity, we recall some definition we will use concerning lower bounded quadratic forms $(\A,D(\A))$, $(\B,D(\B))$ and their associated self-adjoint operators $(A,D(A))$, $(B,D(B))$ on a Hilbert space $h$ (see \cite{Kato}):\\
i) $(\A,D(\A))$ is {\it $\varepsilon$-bounded w.r.t.} $(\B,D(\B))$ for $\varepsilon>0$, if $D(\B)\subseteq D(\A)$ and $\A[\xi]\le \varepsilon\cdot\B[\xi]+b_\varepsilon\cdot\|\xi\|^2$ for some $b_\varepsilon\ge 0$ and all $\xi\in D(\B)$; the infimum of all such $\varepsilon$ is the {\it form bound} of  $(\A,D(\A))$  w.r.t.  $(\B,D(\B))$;\\
ii) $(\A,D(\A))$ is {\it small (resp. infinitesimally small) w.r.t.} $(\B,D(\B))$ if its form bound is strictly less than one (resp. vanishes);\\
iii) $(A,D(A))$ is {\it $\varepsilon$-bounded w.r.t.} $(B,D(B))$ for $\varepsilon>0$, if $D(B)\subseteq D(A)$ and $\|A\xi\|^2\le\varepsilon\cdot \|B\xi\|^2+b_\varepsilon\cdot \|\xi\|^2$ for some $b_\varepsilon\ge 0$ and all $\xi\in D(B)$; the infimum of all such $\varepsilon$ is the {\it operator bound} of  $(A,D(A))$ w.r.t. $(B,D(B))$;\\
iv) $(A,D(A))$ is {\it small (resp. infinitesimally small) w.r.t.} $(B,D(B))$ if its operator bound is strictly less than one (resp. vanishes);\\
v) $(\A,D(\A))$ is said {\it an infinitesimal perturbation of} $(\B,D(\B))$ if $D(\B)\subseteq D(\A)$ and $(\A-\B,D(\B))$ is infinitesinally small w.r.t. $(\B,D(\B))$;\\
vi) $(A,D(A))$ is said {\it infinitesimally perturbation of} $(B,D(B))$ if $D(B)\subseteq D(A)$ and $(A-B,D(B))$ is infinitesinally small with respect to $(B,D(B))$;\\
It is well known that iii) implies i), iv) implies ii) and vi) implies v);\\
vii) $(A,D(A))$ has {\it purely discrete spectrum} if this is made by discrete eigenvalues only (isolated eigenvalues of finite degeneracy); by the Min-Max Theorem this holds true if and only if $(\A,D(\A))$ is {\it proper} in the sense that $\{\xi\in D(\A): \|\xi\|\le 1,\A[\xi]\le 1\}$ is relatively compact in $h$. 

\begin{thm} (Coercivity)\\
Assume $(\tilde q_X^\lambda,D(X)\cap D(X^*))$ to be lower bounded and closable, denote by $(q^\lambda_X,D(q^\lambda_X))$ its closure and by $(Q^\lambda_X,D(Q^\lambda_X))$ the associated lower bounded, self-adjoint operator. Then
\vskip0.2truecm\noindent
i) $(\tilde{\mathcal{Q}}^\lambda_X,\tilde\F_X)$ is lower bounded, closable and its closure $({\mathcal{Q}}^\lambda_X,D({\mathcal{Q}}^\lambda_X))$ bounds the Dirichlet form
\begin{equation}
{\mathcal{Q}}^\lambda_X[\eta]\le \E^\lambda_X[\eta]\qquad \eta\in \F^\lambda_X\subseteq D({\mathcal{Q}}^\lambda_X);
\end{equation}
if moreover, the self-adjoint operator associated to $({\mathcal{Q}}^\lambda_X,D({\mathcal{Q}}^\lambda_X))$ has discrete spectrum, then the spectrum of the self-adjoint operator
$(H^\lambda_X,D(H^\lambda_X))$ associated to $(\E^\lambda_X,\F^\lambda_X)$ is discrete too.
\vskip0.2truecm\noindent
ii) $(Q^\lambda_X,D(Q^\lambda_X))$ is affiliated to $M$, $(j(Q^\lambda_X),JD(Q^\lambda_X))$ is affiliated to $M'$ and $D(Q^\lambda_X)\cap JD(Q^\lambda_X)$ is dense in $L^2(M)$.
\vskip0.2truecm\noindent
Assume now on $D(X)=D(X^*)$, $D(X^*X)=D(XX^*)$ and the quadratic form $(q_0, D(q_0))$ to be infinitesimally small with respect to $(t_X,D(t_X))$. Then
\vskip0.2truecm\noindent
iii)  the form $(\tilde q_X^\lambda,D(X))$ is lower bounded, closed and $(Q^\lambda_X,D(Q^\lambda_X))$ equals the Friedrichs extension of the lower bounded, densely defined, symmetric operator
\begin{equation}
\begin{split}
D(N^\lambda_X)&:=D(X^*X)=D(XX^*)\\
 N^\lambda_X&:=(\lambda-\lambda^{-1})^2\cdot X^*X+(\lambda^{-2}-1)\cdot[X,X^*];
 \end{split}
\end{equation}
iv) in particular, the conclusions in iii) subsist if the self-commutator $([X,X^*],\\D(X^*X))$ is infinitesimally small w.r.t. $(X^*X,D(X^*X))$ and in this case
\begin{equation}
(Q^\lambda_X,D(Q^\lambda_X))=(N^\lambda_X,D(X^*X)).
\end{equation}
If moreover the spectrum of $(X^*X,D(X^*X))$ is discrete, then the spectrum of the generator $(H^\lambda_X,D(H^\lambda_X))$ of the Dirichlet form is discrete too.

\end{thm}
\begin{proof}
i) Since $(\tilde q_X^\lambda,D(X)\cap D(X^*))$ is lower bounded and closable and $J$ is isometric, the quadratic form $J(D(X)\cap D(X^*))\ni\eta\mapsto\tilde q_X^\lambda[J\eta]$ is densely defined, lower bounded and closable too. This implies that $(\tilde{\mathcal{Q}}^\lambda_X,\tilde\F_X)$ is lower bounded and closable as a sum of forms sharing these same properties. The lower bound (3.10) follows from (3.9) and the lower boundedness of $(\tilde{\mathcal{Q}}^\lambda_X,\tilde\F_X)$. The last assertion concerning discreteness of spectra follows from the Min-Max Theorem.\\
ii) Since $(X,D(X))$ and $(X^*,D(X^*))$ are closed operators affiliated to $M$, it follows that for any unitary $u'\in M'$ we have $u'(D(X)\cap D(X^*))\subset D(X)\cap D(X^*)$ and
$\tilde q_X^\lambda[u'\eta]=\tilde q_X^\lambda[\eta]$ for any $\eta\in D(X)\cap D(X^*)$. By approximation, these invariance still hold true for the closure $(q^\lambda_X,D(q^\lambda_X))$ and implies that for all unitaries $u'\in M'$ and all $\eta\in D(Q^\lambda_X)$ one has $u'\eta\in D(Q^\lambda_X)$ and $Q^\lambda_Xu'\eta=u' Q^\lambda_X\eta$. Hence
$(Q^\lambda_X,D(Q^\lambda_X))$ is affiliated to $M$, $(j(Q^\lambda_X),D(j(Q^\lambda_X)))$ is affiliated to $M'$, the operators strongly commute and have a common dense core.\\
iii) Since $D(X)=D(X^*)$ and $(q_0,D(X)\cap D(X^*))=(q_0,D(X))$ is infinitesimally small with respect to $(t_X,D(X))$, the sum
$\tilde q_X^\lambda=(\lambda-\lambda^{-1})^2\cdot t_X+(\lambda^{-2}-1)\cdot q_0$ is lower bounded and closed since $(t_X,D(X))$ is lower bounded and closed.\\
Since $q_0=t_{X^*}-t_X$ is infinitesimally small with respect to $t_X$ on the common domain $D(X)=D(X^*)$, we have that $t_{X^*}$ is relatively bounded with respect to $t_X$ 
and that $t_{X}$ is relatively bounded with respect to $t_{X^*}$. As $D(X^*X)=D(XX^*)$ by assumption, the symmetric operator $(N^\lambda_X,D(N^\lambda_X))$ is densely defined and lower bounded since its quadratic form is the restriction of the lower bounded form $(\tilde q_X^\lambda,D(X))$ to $D(X^*X)$, i.e. $(\eta|N^\lambda_X\eta)=\tilde q_X^\lambda[\eta]$ for all $\eta\in D(X^*X)$. Since $D(X^*X)$ is form core for $(t_X,D(X))$ and $(\tilde q_X^\lambda,D(X))$ is an infinitesimal perturbation of a multiple of it, $D(X^*X)$ is a form core for
$(\tilde q_X^\lambda,D(X))$ too. Since, by definition, the Friedrichs extension of $(N^\lambda_X,D(X^*X))$ is the self-adjoint operator associated to the closure of its quadratic form $(\tilde q_X^\lambda,D(X^*X))$, it results that $(Q^\lambda_X,D(Q^\lambda_X))$ coincides with it.\\
iv) In this case the operator $(N^\lambda_X,D(N^\lambda_X))$ is an infinitesimal symmetric perturbation of a multiple of the self-adjoint operator $(X^*X,D(X^*X))$ and it is self-adjoint by the Kato-Rellich Theorem. Since it is also lower bounded, it has to coincides with its Friedrichs extension $(Q^\lambda_X,D(Q^\lambda_X))$.\\
To prove the last assertion, recall that the spectrum of a lower bounded self-adjoint operator is discrete if and only if its associated quadratic form is proper (see \cite{d}). Now, by a general corollary of the Min-Max Theorem, if the spectrum of $(\lambda-\lambda^{-1})^2X^*X$ is discrete, then the spectrum of $N^\lambda_X$ is discrete too, as the latter operator is the sum of the former and the lower bounded self-adjoint operator $(\lambda^{-2}-1)[X,X^*]$, all with domain $D(X^*X)$. Hence, the lower bounded, closed quadratic form  $(\tilde q_X^\lambda,D(X))$  of $(N^\lambda_X,D(X^*X))$ is a proper functional and consequently the lower bounded, closed form $({\mathcal{Q}}^\lambda_X,D({\mathcal{Q}}^\lambda_X))$ is proper too, as a sum of proper functionals. The lower bound (3.10) then implies that the Dirichlet form is a proper functional.
\end{proof}
\newpage
\section{Superboundedness of a class of semigroups on\\ type I von Neumann algebras}
In this section we introduce a further continuity property, called {\it superboundedness}, for positivity preserving semigroups on standard forms of $\sigma$-finite von
Neumann algebras, showing that the property is owned by a class of semigroups on type I$_\infty$ factors. Also we show how this property persists under domination of positivity preserving semigroups.
\vskip0.2truecm\noindent
As usual, $i_0:M\to L^2(M)$ denotes the symmetric embedding of a $\sigma$-finite von Neumann algebra $M$ endowed with a faithful normal state $\o_0\in M_{*+}$ represented by $\xi_0\in L^2_+(M)$.
\begin{defn}(Excessive vectors and superboundedness)
i) The vector $\xi_0\in L^2_+(M)$ is $(\gamma_0,t_0)$-{\it excessive} or {\it excessive}, for some $\gamma_0, t_0\ge 0$, with respect to a positivity preserving semigroup $\{T_t:t\ge 0\}$ on $L^2(M)$ if the maps $e^{-\gamma_0 t}T_t$ are Markovian w.r.t. $\xi_0$ for any $t>t_0$.\\
Markovian semigroups are just those for which $\xi_0$ is $(0,0)$-excessive;
\vskip0.2truecm\noindent
ii) a positivity preserving semigroup $\{T_t:t\ge 0\}$ is {\it superbounded} if for some $\gamma_0, t_0\ge 0$
\vskip0.2truecm\noindent
a) $\xi_0\in L^2_+(M)$ is $(\gamma_0,t_0)$-{\it excessive},
%i.e. for some $\gamma\ge 0$, $e^{-\gamma t}\cdot T_t:L^2(M)\to L^2(M)$ is Markovian with respect to $\xi_0\in L^2_+(M)$ for all $t> t_0$
\vskip0.2truecm\noindent
b) $T_t(L^2(M))\subseteq i_0(M)$ for all $t>t_0$.
\end{defn}
\noindent
If we endow the subspace $i_0(M)\subseteq L^2(M)$ by the norm of the von Neumann algebra, i.e. $\|i_0(x)\|_M :=\|x\|_M$ for $x\in M$, then superboundedness implies the boundedness of $T_t$ as a map from $(L^2(M),\|\cdot \|_2)$ to $(i_0(M),\|\cdot\|_M)$ for all $t>t_0$. In fact, by the norm continuity of the symmetric embedding $i_0:M\to L^2(M)$, the norm $\|\cdot\|_M$ is stronger than the Hilbert norm $\|\cdot\|_2$ so that the continuous maps $T_t:L^2(M)\to L^2(M)$ are closed when considered from the Hilbert space $L^2(M)$ to the Banach space $(i_0(M),\|\cdot\|_M)$ and,  by the Closed Graph Theorem, they result to be bounded (notice that this involves only condition b) in Definition 4.1).\\
We shall refer to part b) of superboundedness writing $\|T_t\|_{L^2(M)\to M}<+\infty$ for all $t>t_0$ and to part b) of {\it supercontractivity} writing $\|T_t\|_{L^2(M)\to M}\le 1$ for all $t>t_0$.\\
By the Markovianity of $e^{-\gamma_0 t} T_t$ required in i), bounded, positivity preserving maps $S_t:M\to M$ satisfying the relations $i_0(S_t(x))=T_t(i_0(x))$ for $x\in M$ are well defined and one has,
for suitable scalars $b_t\ge 0$,

\[
\|S_t\|\le e^{\gamma_0 t},\qquad \|S_t x\|_M\le b_t\cdot\|i_0(x)\|_{L^2(M)}\qquad x\in M,\quad t>t_0.
\]
%In ii), one essentially requires the existence of $t_0\ge 0$ such that $T_t(L^2(M))\subseteq i_0(M)$ for all $t>t_0$ and such that, if $i_0(y)=T_t(i_0(x))$ for some $x,y\in M$ and $t>t_0$, then one has $\|y\|_M^2\le b_t\cdot\|i_0(y)\|_{L^2(M)}^2=b_t\cdot (y\xi_0|Jy{\color{red}^*}\xi_0)$.
\vskip0.2truecm\noindent
Consider the noncommutative spaces $L^p(M,\o_0)$ for $p\in [2,+\infty]$ defined by the symmetric embedding $i_0:M\to L^2(M)$ (see \cite{ko}). By complex interpolation it follows that a superbounded semigroup is {\it hypercontractive} too in the sense that there exists $T_0\ge 0$ such that $T_t$ is bounded from $L^2(M)$ to $L^4(M,\o_0)$ for $t>T_0$.
\vskip0.2truecm
The following observation will be useful later on.

\begin{lem}(Superboundedness by domination)
Let $\{e^{-tG_0}:t\ge 0\}$ be a superbounded semigroup on $L^2(M)$ such that, for some $\gamma_0, t_0\ge 0$
\[
\xi_0\in L^2_+(M)\,\, \text{is $(\gamma_0,t_0)$-excessive}.
\]
Let $\{e^{-tG_1}:t\ge 0\}$ be a $C_0$-continuous, self-adjoint, positivity preserving semigroup such that, for some $\gamma_1, t_1\ge 0$
\[
\xi_0\in L^2_+(M)\,\, \text{is $(\gamma_1,t_1)$-excessive.}
\]
\vskip0.1truecm\noindent
If the semigroup $\{e^{-tG_1}:t\ge 0\}$ is dominated by the semigroup $\{e^{-tG_0}:t\ge 0\}$ in the sense
\begin{equation}
e^{-tG_1}\eta\le e^{-tG_0}\eta\qquad \eta\in L^2_+(M),\quad t\ge 0,
\end{equation}
then $\{e^{-tG_1}:t\ge 0\}$ is superbounded with
\begin{equation}
\|e^{-tG_1}\eta\|_M\le\|e^{-tG_0}\|_{L^2(M)\to M}\cdot\|\eta\|_2,\qquad \eta\in L^2_+(M),\quad t>t_0\vee t_1.
\end{equation}
%and if $\{e^{-tG_1}:t\ge 0\}$ is ergodic then $\{e^{-tG_0}:t\ge 0\}$ is ergodic too.
\end{lem}
\begin{proof}
The superboundedness of $\{e^{-tG_0}:t\ge 0\}$ and the domination (4.1) imply that $e^{-tG_1}(L^2_+(M))\subset i_0(M_+)$ for any $t>t_0\vee t_1$. Since $L^2_+(M)$ linearly generates $L^2(M)$, it follows that $e^{-tG_1}(L^2(M))\subseteq i_0(M)$ for all $t>t_0\vee t_1$ so that $\{e^{-tG_1}:t\ge 0\}$ is superbounded. The bound (4.2) follows from the domination (4.1) and the supeboundedness of $\{e^{-tG_0}:t\ge 0\}$.
%Finally, if $\{e^{-tG_1}:t\ge 0\}$ is ergodic (see \cite{7} Definition 4.2) then for any fixed $\eta,\zeta\in L^2_+(M)$ there exists $t\ge 0$ such that $(\eta|e^{-tG_1}\zeta)>0$. By (4.1), this implies that $(\eta|e^{-tG_0}\zeta)>0$ so that $\{e^{-tG_0}:t\ge 0\}$ is ergodic.
\end{proof}

\subsection{A class of superbounded Markovian semigroups on a type I$_\infty$ factor}
Let $h$ be a Hilbert space and consider the type I factor $M:=B(h)$. Its (Hilbert-Schmidt) standard representation acts, by left composition, on the space $L^2(M)=L^2(h)$ of Hilbert-Schmidt operators on $h$, where the standard cone $L^2_+(M)=L^2_+(h)$ is that of operators in $L^2(h)$. The modular involution is given by the operator adjoint: $J\xi:=\xi^*$ for $\xi\in L^2(h)$ and the right representation of $B(h)$ on $L^2(h)$ is given by right composition.
\vskip0.2truecm\noindent
Let $H_0$ be a lower bounded, self-adjoint operator affiliated to $B(h)$ (i.e. any self-adjoint, lower bounded operator on $h$) and consider the strongly continuous semigroup on $L^2(h)$ given by
\[
T_t\eta=e^{-tH_0}J(e^{-tH_0}J(\eta))=e^{-tH_0}\circ\eta\circ e^{-tH_0}\qquad \eta\in L^2(h).
\]
Its self-adjoint generator $G_0$ on $L^2(h)$, defined by $G_0(\xi):=\lim_{t\to 0} t^{-1}(\xi-T_t\xi)$ on the subspace $D(G_0)\subset L^2(h)$ for whose vectors the limit exists, coincides with the generalized sum $H_0\dot{+}JH_0J$ (see \cite{Kato}) of the closed operators $H_0$ and $JH_0J$, affiliated to the commuting von Neumann algebras given by the left and right representations of $B(h)$ on $L^2(h)$ (see Lemma 7.1 in Appendix). The operator $H_0$, resp. $JH_0J$, is considered here as acting on a suitable dense subspace of the Hilbert-Schmidt space $L^2(h)$ by left, resp. right, composition. For example, $G_0(\xi)=\overline{H_0\circ\xi} + \overline{\xi\circ H_0}\in L^2(h)$ for those $\xi\in L^2(h)$ such that the operators $H_0\circ \xi$ and $\xi\circ H_0$ are densely defined, closable and bounded on their domains and their closures are Hilbert-Schmidt operators. To ease notation, the operators $H_0\circ \xi, \xi\circ H_0$ will be represented by the juxtaposition $H_0\xi,\xi H_0$ of the symbols of the operators $H_0$ and $\xi$ so that, the formula above appears $G_0(\xi)=\overline{H_0\xi}+\overline{\xi H_0}$. For further details on Hilbert-Schmidt standard form we refer to \cite{cfl} Section 2.
\begin{lem}
If $H_0$ has discrete spectrum ${\rm Sp}(H_0):=\{\lambda_j:j\in\mathbb{N}\}$\footnote{$\N=\{0,1,\cdots\}$} with the increasing eigenvalues written with repetitions according to the their multiplicity, then
\vskip0.2truecm\noindent
i) $G_0$ has discrete spectrum too given by ${\rm Sp}(G_0):=\{\lambda_j+\lambda_k\in\R:(j,k)\in\mathbb{N}\times\mathbb{N}\}$;\\
ii) if $n_{H_0}(\lambda):=\natural\{j\in\mathbb{N}:\lambda_j\le \lambda\}$ is the eigenvalue counting function of $H_0$, then the eigenvalue counting function of $G_0$ is bounded by $n_{G_0}(\lambda)\le (n_{H_0}(\lambda-\lambda_0))^2$, $\lambda\in\R$.
%In particular $G_0\ge 2\lambda_0$.
\end{lem}
\begin{proof}
Let $H_0=\sum_{k=0}^\infty \lambda_k P_k$ be the spectral decomposition of $H_0$ as an operator acting on $h$. Then the spectral decomposition
of $G_0$ is given by
\[
G_0=\sum_{j,k=0}^\infty (\lambda_j+\lambda_k) P_j JP_kJ,
\]
since $\{P_j JP_kJ:j,k\ge 0\}$ is a complete family of mutually orthogonal projections acting on the standard Hilbert space $L^2(h)$ such that
\[
\begin{split}
(H_0+JH_0J)P_jJP_kJ&=H_0P_jJP_kJ+P_jJH_0P_kJ=\lambda_jP_jJP_kJ+\lambda_k P_j P_k J \\
&=(\lambda_j+\lambda_k)P_j JP_kJ.
\end{split}
\]
Thus $G_0$ has the discrete spectrum indicated in the statement and since
$\lambda_j+\lambda_k\le \lambda$ implies both $\lambda_j+\lambda_0\le \lambda$ and $\lambda_0+\lambda_k\le \lambda$, the bound $n_{G_0}(\lambda)\le n_N(\lambda-\lambda_0)^2$ holds true for $\lambda\in\R$.
\end{proof}
Suppose now the lower bounded, self-adjoint operator $H_0$ on $h$ to have a discrete spectrum ${\rm Sp}(H_0):=\{\lambda_j:j\in\mathbb{N}\}$ such that, for some $\beta>0$,
\[
{\rm Tr}(e^{-\beta H_0})=\sum_{k=0}^\infty e^{-\beta\lambda_k}<+\infty,
\]
so that the Gibbs state on $B(h)$ with density matrix $\rho_\beta:=e^{-\beta H_0}/{\rm Tr}(e^{-\beta H_0})$ is well defined
\[
\o_\beta(x):={\rm Tr}(x\rho_\beta)\qquad x\in B(h)
\]
and its representative positive vector is given by $\xi_0:=\rho_\beta^{1/2}\in L^2_+(h)$. Recall that in this case the symmetric embedding $i_0:B(h)\to L^2(h)$ is given by $i_0(x)=\rho_\beta^{1/4}x\rho_\beta^{1/4}$ for $x\in
B(h)$.
\begin{thm}
i) The $C_0$-continuous, self-adjoint semigroup $\{e^{-tG_0}:t>0\}$ is positive preserving and $\xi_0:=\rho_\beta^{1/2}\in L^2_+(h)$ is $(-2(\lambda_0\wedge 0),0)$-excessive;
\vskip0.1truecm\noindent
ii) the semigroup $\{e^{-tG_0}:t>0\}$ is superbounded with
%in the sense that
%\[
%e^{-tG_0}(L^2(h))\subseteq i_0(B(h))\qquad t>\beta/4
%\]
%and $\|x\|_{B(h)}\le \|\xi\|_{L^2(h)}$ if $\xi\in L^2({}\color{red}h)$ and $x\in B(h)$ are related by $i_0(x)=e^{-tG_0}\xi$.
\[
\|e^{-tG_0}\|_{L^2(M)\to M}\le e^{-(2t-\beta/2)\lambda_0}\qquad t>\beta/4.
\]
In particular, if $\lambda_0\ge 0$, the semigroup is Markovian and supercontractive.
\end{thm}
\begin{proof}
Replacing $H_0$ with $H_0+\beta^{-1}\ln {\rm Tr}(e^{-\beta H_0})$, we may just consider the case ${\rm Tr}(e^{-\beta H_0})=1$.
\vskip0.1truecm\noindent
i) If $\xi\in L^2_+(h)$, since $e^{-tH_0}$ is self-adjoint, we have $e^{-tG_0}\xi=(e^{-tH_0})^*\xi e^{-tH_0}\in L^2_+(h)$ for any $t\ge 0$, showing that the semigroup is positivity preserving. Since
$\lambda_0\le H_0$ and $\beta>0$, for any $t\ge 0$ we have
\[
\begin{split}
e^{-tG_0}\xi_0&=e^{-tH_0}\rho_\beta^{1/2}e^{-tH_0}=e^{-\beta H_0/4}e^{-2tH_0}e^{-\beta H_0/4}\\
&\le e^{-2t\lambda_0}e^{-\beta H_0/4}e^{-\beta H_0/4}= e^{-2t\lambda_0}\xi_0
\end{split}
\]
so that $\xi_0$ is $(-2(\lambda_0\wedge 0),0)$-excessive.\\
ii) For $\xi\in L^2(h)$ and $t>\beta/4$ we have $x:=e^{-(t-\beta/4)H_0}\xi e^{-(t-\beta/4)H_0}\in B(h)$
%, $\|x\|_{B(h)}\le \|\xi\|_{B(h)}$ 
and
\[
\begin{split}
i_0(x)&=\rho_\beta^{1/4}x\rho_\beta^{1/4}=e^{-\beta H_0/4}e^{-(t-\beta/4)H_0}\xi e^{-(t-\beta/4)H_0} e^{-\beta H_0/4} \\
&=e^{-tH_0}\xi e^{-tH_0}=e^{-tG_0}\xi.
\end{split}
\]
Since for $t>\beta/4$ we have $\|e^{-(t-\beta/4)H_0}\|_{B(h)}\le e^{-(t-\beta/4)\lambda_0}$, we get
\[
\begin{split}
\|x\|_{B(h)}&\le\|x\|_{L^2(h)}=\|e^{-(t-\beta/4)H_0}\xi e^{-(t-\beta/4)H_0}\|_{L^2(h)} \\
&\le \|e^{-(t-\beta/4)H_0}\|_{B(h)}\|\xi\|_{L^2(h)}\|e^{-(t-\beta/4)H_0}\|_{B(h)} \\
&\le e^{-(2t-\beta/2)\lambda_0}\cdot \|\xi\|_{L^2(h)}.
\end{split}
\]
\end{proof}

\section{General quantum Ornstein-Uhlenbeck semigroups}
In this section we apply the above framework to construct a family of Dirichlet forms and Markovian semigroups, a special case of which is the quantum Ornstein-Uhlenbeck semigroup studied in \cite{cfl}. While in \cite{cfl} we computed explicitly the spectrum of the generator and proved the Feller property with respect to the algebra of compact operators, here we prove, for each semigroups we construct, subexponential spectral growth rate and domination with respect to positivity preserving semigroups belonging to a natural related class (see Appendix 7.1).
\vskip0.2truecm\noindent
On the Hilbert space $h:=l^2(\mathbb{N})$, consider the C$^*$-algebra of compact operators $\mathcal{K}(h)$. The Number Operator $(N,D(N))$, defined by the natural basis $e:=\{e_k\in l^2(\mathbb{N}):k\in\N\}$ as
\[
D(N):=\Bigl\{\sum_{k\in\N}c_k\cdot e_k:\sum_{k\in \N} k^2\cdot |c_k|^2<+\infty\Bigr\}\qquad Ne_k:=ke_k\qquad k\in\mathbb{N},
\]
generates the $C_0$-continuous group of automorphisms $\alpha:=\{\alpha_t\in {\rm Aut}(\mathcal{K}(h)):t\in\R\}$
\[
\alpha_t(B):=e^{itN}Be^{-itN}\qquad B\in\mathcal{K}(h),\quad t\in \R.
\]
For any $\beta>0$ there exists a unique $(\alpha,\beta)$-KMS state $\o_\beta$, satisfying the KMS condition
\[
\o_\beta (A\alpha_{i\beta}(B))=\o_\beta (BA)
\]
for $\alpha$-analytic elements $A,B$, given by, in terms of the density matrix,
\[
\rho_\beta:=(1-e^{-\beta})e^{-\beta N}=(1-e^{-\beta})\sum_{k\in\mathbb{N}}e^{-\beta k} p_k,\quad \o_\beta(A):={\rm Tr}(A\rho_\beta),\quad A\in \mathcal{K}(h)
\]
($p_k$ being the projection onto $\mathbb{C}e_k$).
The von Neumann algebra $M$ generated by the GNS representation of $\o_\beta$
can be identified with $B(h)$ and the normal extension of $\o_\beta$ on it is still given by the formula above for any $A\in B(h)$. The extension of the automorphisms group $\alpha$ to a $C_0^*$-continuous group on $B(h)$ is given by the same formula above on $\mathcal{K}(h)$.\\
In the Hilbert-Schmidt standard form of $M:=B(h)$ described in Section 4.1, the cyclic and separating vector representing $\o_\beta$ is given by
\[
\xi_0:=\rho_\beta^{1/2}=\sqrt{1-e^{-\beta}}e^{-\beta N/2}\in L^2_+(h).
\]
The action of the Hilbert algebra unbounded conjugation operator $S_0$ on $L^2(h)$, characterized as $S_0(x\xi_0):=x^*\xi_0$ for $x\in B(h)$, can be identified on a suitable domain $D(S_0)\subset L^2(h)$ with
\[
S_0(\eta)=\overline{\rho_\beta^{-1/2}\eta^*\rho_\beta^{1/2}}
\]
and its polar decomposition $S_0=J\Delta_0^{1/2}$ is provided by the modular operator
\[
\Delta_0^{1/2}(\eta)=\overline{\rho_\beta^{1/2}\eta\rho_\beta^{-1/2}}=\overline{e^{-\beta N/2}\eta e^{\beta N/2}}\qquad \eta\in D(S_0).
\]
The modular group of $\o_\beta$, satisfying the modular condition $\o_\beta(A\sigma^{\o_\beta}_{-i}(B))=\o_\beta (BA)$, for analytics elements $A,B$, is then given by $\sigma^{\o_\beta}_t=\alpha_{-\beta t}$ for $t\in\R$.
Regarding the Number Operator $N$ as an operator affiliated to $B(h)$ in its normal representation on $L^2(h)$ (i.e. acting, on a suitable domain of the Hilbert-Schmidt operators, by left composition), we have that the modular (Araki) Hamiltonian is given by the strong sum of the densely defined, self-adjoint operators $N$ and $-JNJ$ (belonging to commuting von Neumann algebras)
\[
-\ln\Delta_0=\beta\overline{N-JNJ}
\]
and its (discrete) spectrum is given by ${\rm Sp}(-\ln\Delta_0)=\beta\mathbb{Z}$. Consequently ${\rm Sp}(\Delta_0^{1/2})=e^{\beta\mathbb{Z}/2}$
with uniform multiplicity one.
\vskip0.2truecm\noindent
Let us consider the annihilation and creation operators $(A,D(A)), (A^*,D(A^*))$ on $h$, defined on the domain $D(A):=D(\sqrt N)=:D(A^*)$ as
\[
Ae_0:=0,\quad Ae_k:=\sqrt{k}e_{k-1}\quad\text{if $k\ge 1$},\qquad A^*e_k:=\sqrt{k+1}e_{k+1}\qquad k\in\N.
\]
They satisfy the Canonical Commutation Relations $AA^*=A^* A+I$, as closed operators defined on $D(N$), and allow to represent the Number Operator as $N=A^*A$. All these operators and their functional calculi are understood as affiliated to $B(h)$ acting by left composition on operators belonging to the Hilbert-Schmidt class $L^2(h)$.
\vskip0.2truecm\noindent
Let us consider the family of operators affiliated to $B(h)$
\begin{equation}
D(X_m)=D(N^{m/2})\qquad X_m:=(A^*)^m\qquad m\in\mathbb{N}\setminus\{0\}.
\end{equation}
%\begin{equation}
%X_m:=\begin{cases}
%(A^*)^m&\,\,\,\,\,m\in\mathbb{N}\setminus\{0\}\\
%A^{-m} &-m\in\mathbb{N}\setminus\{0\}.
%\end{cases}
%\end{equation}
\begin{lem}
i) For any $m\ge 1$ and $\lambda_m^2:=e^{-m\beta/2}$ we have
\begin{equation}
\begin{split}
D(X_m)=&D(X_m^*)=D(N^{m/2}),\quad D(X_m^*X_m)= D(X_mX_m^*)=D(N^m)\\
X_{m}^*X_{m}&=A^{m}(A^*)^{m}=(N+m)(N+m-1)\cdots (N+2)(N+1)\\
X_mX_m^*&= (A^*)^{m}A^{m} = N(N-1)(N-2)\cdots (N-(m-1))
 \end{split}
\end{equation}
and the self-commutator $([X_m,X^*_m], D(N^{m}))$ is a self-adjoint operator, infinitesimally small with respect to $(X_m^*X_m,D(N^m))$ and $(N^m,D(N^m))$;\\
ii) $X_m\xi_0=(A^*)^m\xi_0\in L^2(h)$ is an eigenvector of $\Delta_0^{1/2}$ corresponding to the eigenvalue $\lambda_{m}^{2}$, \\
$X_m^*\xi_0=(A)^m\xi_0\in L^2(h)$ is an eigenvector of $\Delta_0^{1/2}$ corresponding to the eigenvalue $\lambda_{m}^{-2}$.
% and having norm
%\[
%\|(A^*)^m\xi_0\|^2=(\xi_0|A^m(A^*)^m\xi_0)=................
%\]
%iii) For any $m\ge 1$, $A^m\xi_0\in L^2(h)$ is an eigenvector of $\Delta_0^{1/2}$ corresponding to the eigenvalue $e^{m\beta/2}=\lambda_m^2$ having norm
%\[
%\|A^m\xi_0\|^2=(\xi_0|(A^*)^m A^m\xi_0)=......................
%\]
\end{lem}
\begin{proof}
i) Formulae (5.2) follow by induction starting from the case $m=1$. They show that the self-commutator is a polynomial in $N$ of degree $(m-1)$ and this implies the remaining conclusion. ii) Since $A^*e_k:=\sqrt{k+1}e_{k+1}$  for $k\in\mathbb{N}$, we have $(A^*\xi_0)e_k=A^*(\rho_\beta^{1/2}(e_k))=(1-e^{-\beta})^{1/2}e^{-\beta k/2}A^*e_k=(1-e^{-\beta})^{1/2}e^{-\beta k/2}\sqrt{k+1}e_{k+1}$
and then for, any $m\in\mathbb{N}$, we have too $((A^*)^m\xi_0)e_k=(1-e^{-\beta})^{1/2}e^{-\beta k/2}(A^*)^m e_k=(1-e^{-\beta})^{1/2}e^{-\beta k/2}\sqrt{k+1}\cdots \sqrt{k+m}\,e_{k+m}$ so that
\[
\begin{split}
(\Delta_0^{1/2}((A^*)^m\xi_0))e_k&=(\rho_\beta^{1/2} ((A^*)^m\rho_\beta^{1/2})\rho_\beta^{-1/2})e_k=\rho_\beta^{1/2}((A^*)^me_k)\\
&=\sqrt{k+1}\cdots\sqrt{k+m}\,\rho_\beta^{1/2}e_{k+m}\\
&=(1-e^{-\beta})^{1/2}e^{-\beta (k+m)/2}\sqrt{k+1}\cdots\sqrt{k+m}\,e_{k+m}\\
&=e^{-m\beta/2}((A^*)^m\xi_0)e_k.
\end{split}
\]
Hence $(A^*)^m\xi_0\in L^2(h)$ is eigenvector of $\Delta_0^{1/2}$ corresponding to the eigenvalue $\lambda_m^2:=e^{-m\beta/2}$. The other series of eigenvalues follow from Lemma 2.9.
%from the identity $X_m^*\xi_0=S_0(X_m\xi_0)$.
\end{proof}
\noindent
We are now in position to apply Theorem 2.5 with $Y=X_m$, $\lambda=e^{-m\beta/4}$, $\xi_0=\rho_\beta^{1/2}\in L^2(l^2(\N))$\\ and consider the Dirichlet form
$(\E^{\lambda_m}_{X_m},\F^{\lambda_m}_{X_m})$ on $L^2(l^2(\N))$ and its generator $(H^{\lambda_m}_{X_m},D(H^{\lambda_m}_{X_m}))$.\\
The following result generalizes, in particular, some of those obtained in \cite{cfl} for the quantum Ornstein-Uhlebeck semigroup, corresponding to the present parameter $m=1$.
\begin{thm} (Spectral growth rate)
For $m\ge 1$ and $\lambda_m^2:=e^{-m\beta/2}$, the operator $(H^{\lambda_m}_{X_m},D(H^{\lambda_m}_{X_m}))$ has discrete spectrum and subexponential spectral growth rate
\[
{\rm Tr}(e^{-tH_{X_m}^{\lambda_{m}}})<+\infty\qquad t>0.
\]
\end{thm}
\begin{proof}
By Lemma 5.1 i) above, the self-adjoint operator
\[
N_{X_m}^{\lambda_{m}}:=(\lambda_m-\lambda_m^{-1})^2 X_m^* X_m + (\lambda_m^{-2}-1)[X_m, X_m^*]
\]
has, on its domain, the following explicit form
\begin{equation}
\begin{split}
N_{X_m}^{\lambda_{m}}&=(\lambda_m^2-1)A^m (A^*)^m + (\lambda_m^{-2}-1)(A^*)^m A^m\\
&=(\lambda_m^2-1)(N+1)\cdots (N+m) + (\lambda_m^{-2}-1)N(N-1)\cdots (N-(m-1)) \\
&= (\lambda_m^2+\lambda_m^{-2}-2)N^m +p_{m-1}(N)= (\lambda_m-\lambda_m^{-1})^2N^m +p_{m-1}(N),
\end{split}
\end{equation}
where $p_{m-1}:\R\to\R$ is a suitable polynomial of degree $(m-1)$ with real coefficients. Since for any $\varepsilon\in (0,1)$ one has $b_m^\varepsilon:=\inf_{s\ge 0}(\varepsilon (\lambda_m-\lambda_m^{-1})^2s^m+p_{m-1}(s))>-\infty$ and $(\lambda_m-\lambda_m^{-1})^2=(2\sinh(m\beta/4))^2>0$,
$(N_{X_m}^{\lambda_{m}},D(N^m))$ is lower bounded, self-adjoint with subexponential spectral growth rate
\[
\begin{split}
{\rm Tr}(e^{-tN_{X_m}^{\lambda_{m}}})&\le e^{-tb_m^\varepsilon}{\rm Tr}(e^{-t(1-\varepsilon)(\lambda_m-\lambda_m^{-1})^2N^m}) \\
&=e^{-tb_m^\varepsilon}\sum_{k=1}^\infty  e^{-t(1-\varepsilon)(\lambda_m-\lambda_m^{-1})^2k^m},\qquad t>0,
\end{split}
\]
by \cite{r} Proposition 1.2.15. Applying Lemma 4.3, these same properties (having discrete spectrum and sub-exponential spectral growth rate) hold true for the sum $N_{X_m}^{\lambda_{m}}\dot{+}j(N_{X_m}^{\lambda_{m}})$. Also, since $D(X_m)=D(X_m^*)=D(N^{m/2})$ and $D(X_m^*X_m)=D(X_mX_m^*)=D(N^m)$, by Theorem 3.3 iii) and iv) with the notations there introduced, we deduce spectrum discreteness and growth rate for $H_{X_m}^{\lambda_m}$ too:
\[
\begin{split}
{\rm Tr}(e^{-tH_{X_m}^{\lambda_{m}}})&\le{\rm Tr}(e^{-t(N_{X_m}^{\lambda_{m}}\dot{+}j(N_{X_m}^{\lambda_{m}})})= {\rm Tr}(e^{-tN_{X_m}^{\lambda_{m}}}Je^{-tN_{X_m}^{\lambda_{m}}}J)\\
&=\Bigl({\rm Tr}(e^{-tN_{X_m}^{\lambda_{m}}})\bigr)^2,\qquad t>0.
\end{split}
\]
\end{proof}

\begin{thm}(Domination)
For $m\ge 1$ and $\lambda_m^2:=e^{-m\beta/2}$, the following properties hold:
\vskip0.1truecm\noindent
the Markovian semigroup $\{e^{-tH_{X_m}^{\lambda_{m}}}:t\ge 0\}$, associated to the Dirichlet form $(\E^{\lambda_m}_{X_m},\F^{\lambda_m}_{X_m})$ dominates the
%$C_0$-continuous, self-adjoint, positivity preserving
Markovian semigroup $\{e^{-tG_1}:t\ge 0\}$, generated  by the closed, self-adjoint operator $(G_1,D(G_1))$ on $L^2(h)$ given by
\[
\begin{split}
D(G_1)&:=D(N^m)\cap JD(N^m)\\
G_1&:=(\lambda_m^2\cdot X_m^*X_m + \lambda_m^{-2}\cdot X_mX_m^*) \dot{+} j(\lambda_m^2\cdot X_m^*X_m + \lambda_m^{-2}\cdot X_mX_m^*),
\end{split}
\]
which can be expressed as
\[
e^{-tG_1}(\eta)=e^{-tB_m}j(e^{-tB_m})(\eta)=e^{-tB_m}\eta e^{-tB_m}\qquad \eta\in L^2(M),
\]
by the self-adjoint, positive operator $(B_m,D(B_m)):=(\lambda_m^2\cdot X_m^*X_m + \lambda_m^{-2}\cdot X_mX_m^*,D(N^m))$, affiliated to $M:=B(h)$ in its left action on $L^2(M)=L^2(h)$.
%ii) if $W$ is any weight on $M\otimes_{\rm bin}M^\circ$ dominating the weight $W_m$ determined by
%\[
%W_m(x\otimes y):=2\bigl[(X_m^*i_0(x)|j(X_m^*)i_0(y))+(X_m i_0(x)|j(X_m)i_0(y))\bigr],\qquad x\otimes y\in M\otimes_{\rm bin}M^\circ,
%\]
%then the semigroup generated by $H^{\lambda_m}_{X_m}+W$ is superbounded.

%\[
%\|e^{-tH_{A^m}^{\lambda_{m}}}\|_{L^2(h)\to B(h)}\le1\qquad t>\beta/4.
%\]
%iii) the associated Dirichlet form $\E_{A^m}^{\lambda_{m}}$ satisfies a logarithmic Sobolev inequality.
\end{thm}
\begin{proof}
Set $(q_0,D(q_0)):=(\E^{\lambda_m}_{X_m},\F^{\lambda_m}_{X_m})$ and consider the forms $(q_1,D(q_1))$, $(w,D(w))$ given by $D(q_1):=\tilde\F_{X_m}=:D(w)$ and
\[
\begin{split}
q_1[\eta]&:=\lambda_m^2(\|X_m\eta\|^2+\|X_mJ\eta\|^2)) + \lambda_m^{-2}(\|X_m^*\eta\|^2+\|X_m^*J\eta\|^2),\\
w[\eta]&:=2[(X_m\eta |JX_m^*J\eta)+(X_m^*\eta|JX_mJ\eta)],
\end{split}
\]
so that on $D(q_1):=\tilde\F_{X_m}$, the representation (3.1) of the form $(\E^{\lambda_m}_{X_m},\F^{\lambda_m}_{X_m})$ can be written as
\[
q_1[\eta]=q_0[\eta]+w[\eta]\qquad \eta\in D(q_1).
\]
As, by definition, $(q_0,D(q_0))$ is a Dirichlet form, its associated self-adjoint operator
\[
(G_0,D(G_0)):=(H^{\lambda_m}_{X_m},D(H^{\lambda_m}_{X_m}))
\]
generates a Markovian, hence a $C_0$-continuous, self-adjoint, positivity preserving, semigroup $\{e^{-tG_0}:t\ge 0\}$. Since, by definition (see statement and proof of Theorem 2.5) and (5.1), $\tilde\F_{X_m}=D(X_m)\cap D(X_m^*)\cap J(D(X_m)\cap D(X_m^*))=D(N^{m/2})\cap JD(N^{m/2})$, the quadratic form $(q_1,D(q_1))$ is closed and the associated self-adjoint operator is just
$(G_1,D(G_1))$. Since $\{e^{-tG_1}:t\ge 0\}$ is positivity preserving (see Appendix 7.1), to apply Lemma 4.2, we exploit the characterization of domination between positivity preserving semigroups on standard forms of von Neumann algebras, established in \cite{ACS} Theorem 3.1: the semigroup $\{e^{-tG_1}:t\ge 0\}$ is dominated by $\{e^{-tG_0}:t\ge 0\}$ if and only if each one of the following properties is verified:
\vskip0.1truecm\noindent
a) $D(q_1)\subseteq D(q_0)$,
\vskip0.1truecm\noindent
b) $q_0(\eta|\zeta)\le q_1(\eta|\zeta)$ for all $\eta,\zeta\in D(q_1)\cap L^2_+(M)$,
\vskip0.1truecm\noindent
c) if $\eta\in D(q_0)\cap L^2_+(M)$, $\zeta\in D(q_1)\cap L^2_+(M)$ and $\eta\le\zeta$, then $\eta\in D(q_1)$.
\vskip0.1truecm\noindent
Condition a) holds true since $D(q_1):={\tilde\F}_{X_m}\subseteq \F^{\lambda_m}_{X_m}=:D(q_0)$. To prove b), consider the set $C(e)\subseteq L^2(h)$ of all Hilbert-Schmidt operators which are finite linear combination of the partial isometries $\{e_j\otimes e_k^*:j,k\in\N\}$ of the natural basis $e:=\{e_k\in h:k\in\N\}$ and set $C_+(e):=C(e)\cap L^2_+(h)$, $C_\R(e):=C(e)\cap L^2_\R(h)$, where $L^2_\R(h)=L^2_+(h)-L^2_+(h)$ is the self-adjoint part of $L^2(h)$. Since $\{e_j\otimes e_k^*:j,k\in\N\}$ is a Hilbert basis for $L^2(h)$, $C(e)$ is dense in $L^2(h)$. For $A\in L^2_\R(h)$ and $B\in C(e)$ we have $(B^*+B)/2\in C_\R(e)$ and $\|A-(B^*+B)/2\|_2\le \|A-B\|_2$ so that $C_\R(e)$ is dense in $L^2_\R(h)$. For any $B\in C_\R(e)$ we have $B_+\in C_+(e)$ since $C_\R(e)=\bigcup_{j\in\N}L^2_\R(h_j)$, where $h_j:={\rm Lin}\{e_k\in h:k=0,\ldots ,j\}$, and if $B\in L^2_\R(h_j)$  for some $j\in \N$, then $B_+\in L^2_\R(h_j)$. Since the Hilbertian projection of $L^2_\R(h)$ onto $L^2_+(h)$ is a contraction, for any $A\in L^2_+(h)$ and $B\in C_\R(e)$ we have $\|A-B_+\|_2=\|A_+ -B_+\|_2\le \|A-B\|_2$ showing that the cone $C_+(e)$ is dense in the positive cone $L^2_+(h)$.\\
%If $\eta\in L^2_+(h)$ is such that $(\zeta|\eta)_2=0$ for all $\zeta\in C_{00}^+([e])$ then $(e_k\otimes e_k|\eta)_2=0$ for all $k\in \N$, as $e_k\otimes e_k\in C_{00}^+([e])$. Thus $0=(e_k\otimes e_k|\eta)_2=(e_k|\eta(e_k))$ for all $k\in\N$ and $\eta=0$.
It follows from Lemma 5.1 that $C(e)$ is a $J$-invariant core for $(X_m,D(X_m))$ and $(X_m^*,D(X_m^*))$ which is left globally invariant by both operators: $X_m(C(e))\\\subseteq C(e)$, $X_m^*(C(e))\subseteq C(e)$. Let $P_j$ the finite rank projection on $h$
%$L^2(h)$ 
with range $h_j$,
%$L^2(h_j)$, 
for any $j\in\N$. Then if $\eta,\zeta\in C_+(e)$ then $X_m\zeta=P_jX_mP_k\zeta$ and $X_m^*\eta=P_jX_m^*P_k\eta$ for sufficiently large  $j,k\in \N$. Since $P_{j+m}X_m P_j\in B(h)$, $(P_j X_m P_{j+m})J(P_j X_m P_{j+m})J$ is positivity preserving and we have
\[
(X_m^*\eta|JX_mJ\zeta)=(\eta|(P_j X_m P_{j+m})J(P_j X_m P_{j+m})J\zeta)\ge 0.
\]
By the core property, the positivity of $(X_m^*\eta|JX_mJ\zeta)$ extends to any $\eta,\zeta\in D(X_m^*)\cap JD(X_m)$ and an analogous reasoning shows that $(X_m\eta|JX_m^*J\zeta)\ge 0$
is true for any $\eta,\zeta\in D(X_m)\cap JD(X_m^*)$. Since $D(q_1)=\tilde\F_{X_m}$, altogether these properties allows to check b) as follows for $\eta,\zeta\in D(q_1)\cap L^2_+(M)$
\[
q_1(\eta|\zeta)-q_0(\eta|\zeta)=w(\eta|\zeta)=2[(X_m\eta |JX_m^*J\zeta)+(X_m^*\eta|JX_mJ\zeta)].
\]
%Since, by Lemma 5.1, $B_m$ is polynomial function of $N$ and $C(e)$ is generated by the partial isometries of the Hilbert basis of eigenvectors of $N$, it follows that $e^{-tG_1}(C_+(e))\subset C_+(e)$ for all $t\ge 0$. To check that $C(e)\subset D(G_0)$, notice that for $\eta\in D(q_0)$ and $\zeta\in C(e)$ we have
%\[
%q_0(\eta|\zeta)=q_1(\eta|\zeta)-w(\eta|\zeta)=(\eta|G_1(\zeta))-2(\eta|X_m^*\zeta X_m+X_m\zeta X_m^*)
%\]
%so that $|q_0(\eta|\zeta)|\le \|\eta\|_2\cdot \|G_1(\zeta)-2(X_m^*\zeta X_m+X_m\zeta X_m^*)\|_2$. Since $D(q_0)$ is dense in $L^2(h)$ this proves that $\zeta\in D(G_0)$ and that $G_0(\zeta)=G_1(\zeta)-2(X_m^*\zeta X_m+X_m\zeta X_m^*)$. Finally, we may apply Lemma 4.2, since for any $\eta\in C_+(e)$ we have
%\[
%G_1(\eta)-G_0(\eta)=2[X^*_mj(X^*_m)+X_mj(X_m)](\eta)=2(X_m^*\eta X_m + X_m\eta X_m^*)\in L^2_+(h).
%\]
To check c), since $D(q_1):=\tilde\F_{X_m}$ is core for $(\E^{\lambda_m}_{X_m},\F^{\lambda_m}_{X_m})$, let $\eta_n\in D(q_1)$ be a sequence such that
\[
\lim_n \bigl(q_0[\eta_n-\eta]+\|\eta_n-\eta\|_2^2]\bigr)=0.
\]
Let $\eta_n\wedge\zeta:={\rm Proj}(\eta_n,\zeta-L^2_+(M))$ be the Hilbert projection of $\eta_n\in L^2_+(M)$ onto the closed and convex set $\zeta-L^2_+(M)\subset L^2_\R(M)$. Since, by Lemma 4.4 in \cite{c1}, we have $\eta_n\wedge\zeta=\zeta\wedge\eta_n=\eta_n-(\zeta-\eta_n)_-$, the continuity of the Hilbert projections and the fact that
$\eta\le\zeta$, imply
\[
\lim_n\|\eta-\eta_n\wedge\zeta\|_2=\lim_n\|\eta-\eta_n+(\zeta-\eta_n)_-\|_2=\|(\zeta-\eta)_-\|_2=0.
\]
Since $\{e^{-tG_0}:t\ge 0\}$ and $\{e^{-tG_1}:t\ge 0\}$ are positivity preserving, by Proposition 4.5 iii) in \cite{c1} we have
\[
\eta_n\wedge\zeta\in D(q_1),\qquad q_0[\eta_n\wedge\zeta]\le q_0[\eta_n\wedge\zeta]+q_0[\eta_n\vee\zeta]\leq q_0[\eta_n]+q_0[\zeta].
\]
Since $\eta_n\wedge\zeta,\zeta\in D(q_1)$ and, by definition, $\eta_n\wedge\zeta\le\zeta$, we have also (using the property of the quadratic form $w$ established in the proof of b)) 
$w[\eta_n\wedge\zeta]\le w[\zeta]$ so that
\[
q_1[\eta_n\wedge\zeta]=q_0[\eta_n\wedge\zeta]+w[\eta_n\wedge\zeta]\le q_0[\eta_n]+q_0[\zeta]+w[\zeta]=q_0[\eta_n]+q_1[\zeta].
\]
Since the quadratic form $(q_1,D(q_1))$ is closed on $L^2(M)$, it is lower semicontinuous when considered as a functional on $L^2(M)$ taking values in the extended positive half-line $[0,+\infty]$ and it is finite exactly on $D(q_1)$. We then have
\[
q_1[\eta]\le\liminf_n q_1[\eta_n\wedge\zeta]\le\liminf_n\bigl(q_0[\eta_n]+q_1[\zeta]\bigr)=q_0[\eta]+q_1[\zeta]<+\infty
\]
so that $\eta\in D(q_1)$. By \cite{ACS} Theorem 3.1, $\{e^{-tG_1}:t\ge 0\}$ is dominated by $\{e^{-tG_0}:t\ge 0\}$: $e^{-tG_1}\eta\le e^{-tH^{\lambda_m}_{X_m}}\eta$ for all $\eta\in L^2_+(M)$ and $t\ge 0$. Choosing
$\eta:=\xi_0$ one has $e^{-tG_1}\xi_0\le e^{-tH^{\lambda_m}_{X_m}}\xi_0\le\xi_0$ for all $t\ge 0$ so that $\{e^{-tG_1}:t\ge 0\}$ is Markovian.
\end{proof}
%-------------------------------------------------------------------------------------------------------------------------------------------------------------------------------

\section{Dirichlet forms associated to deformations of the CCR relations}
In this section we outline the construction of Dirichlet forms associated to deformations of annihilation and creation operators in the framework and notations of Section 5. To use the tools of Section 2 to this end, we need to represent eigenvectors of (isolated) eigenvalues of the Araki Hamiltonian as in Lemma 2.1.
%by (unbounded) operator affiliated to the Hilbert-Schmidt representation of a type I factor von Neumann algebra.
\subsection{Deformation of the CCR relations}
Let $g:\R\to\R$ be a function vanishing on $(-\infty,0]$, strictly increasing on $[0,+\infty)$ and satisfying, for $\beta >0$ and $\ell\in\N$ to be fixed later,
%\begin{equation}
%\liminf_{n\in\N}(g(n+1)-g(n))>0
%\end{equation}
%so that
\begin{equation}
\sum_{n=0}^\infty n^{\ell} e^{-\beta g(n)}<+\infty.
\end{equation}
Consider the automorphisms group of the $C^*$-algebra of compact operators
\[
\alpha_t(B):=e^{itg(N)}Be^{-itg(N)}\qquad t\in\R,\quad B\in \mathcal{K}(h)
\]
whose Gibbs equilibrium state $\o_\beta(\cdot)={\rm Tr}(\cdot\rho_\beta)$ is represented by the density matrix $\rho_\beta:=e^{-\beta g(N)}/Z(\beta)$ with partition function $Z(\beta):={\rm Tr}(e^{-\beta g(N)})$. Let $\xi_0:=\rho_\beta^{1/2}\in L^2_+(h)$ be the cyclic vector giving rise to the modular group of the normal extension of $\omega_\beta$ to $B(h)$
\[
\sigma_t^{\o_\beta}(B)=\alpha_{-t\beta}(B)=e^{-it\beta g(N)}Be^{it\beta g(N)}\qquad B\in B(h),\quad t\in\R.
\]
Then $\Delta_0^{it}(\eta)=\rho_\beta^{it}\eta\rho_\beta^{-it}$ for all $\eta\in L^2(h)$ and
%\[
%\Delta_0^{1/2}=\rho_\beta^{1/2}j(\rho_\beta^{-1/2})=e^{-\beta g(N)/2}j(e^{\beta g(N)/2})\qquad sp(\Delta_0^{1/2})=\{e^{\beta(g(m)-g(n))/2}: m,n\in\N\}.
%\]
%T
the Araki Hamiltonian is the strong sum
\[
\ln\Delta_0=-\beta\overline{g(N)-j(g(N))}.
\]
Since for each $m,n\in\N$, $\nu_{m,n}:=\beta(g(m)-g(n))$ is an eigenvalue of $\ln\Delta_0$ with eigenvector $e_m\otimes e_n^*\in L^2(h)$ and $\{e_m\otimes e_n^*:m,n\in\N\}$ is a Hilbert basis, the spectrum of $\ln\Delta_0$ is
\[
sp(\ln\Delta_0)=\overline{\{\nu_{m,n}:m,n\in\N\}}.
\]
All eigenvalues are isolated if, for example,
\[
\liminf_{m>n\ge 0}\frac{g(m)-g(n)}{m-n}>0.
\]
\begin{prop}
Suppose $\nu:=\nu_{m,n}\ge 0$ to be an isolated eigenvalue of the Araki Hamiltonian with $m\ge n$, set $\ell:=m-n\in\N$ and let $f\in C^\infty_0(\R)$ be a Schwartz function whose Fourier Transform\footnote{Fourier transform convention: ${\hat f}(s):=\int_\R dt f(t)e^{ist}$.} $\hat f\in C^\infty_0(\R)$ is supported by $[\nu-\varepsilon,\nu+\varepsilon]$ and is strictly positive on $(\nu-\varepsilon,\nu+\varepsilon)$, with ${\hat f}(\nu)=1$, for
\[
0<\varepsilon < {\rm dist}(\nu,sp(\ln\Delta_0)\setminus\{\nu\}).
\]
%Let  $f:\mathbb{C}\to\mathbb{C}$ be an entire function of exponential type
%$|f_{m,n}(z)|\le C\cdot e^{\varepsilon |z-\nu_{m,n}|}$, for some $C\ge 0$ and $0<\varepsilon <{\rm dist}(\nu_{m,n},sp(\ln\Delta_0))$, which is square-integrable over horizontal lines and integrable along the real axis,
%such that ${\hat f}_{m,n}\in C_c(\nu_{m,n}-\varepsilon,\nu_{m,n}+\varepsilon)$ with ${\hat f}_{m,n}(\nu_{m,n})=1$.
%\vskip0.1truecm\noindent
Then, setting $k(t):=\beta(g(t+\ell))-g(t))$ and $p(t):={\hat f}(k(t))$ for $t\in\R$, we have
\vskip0.1truecm\noindent
i) $sp(k(N))\subset sp(\ln\Delta_0)$ and $\nu\in sp(k(N))$ is an isolated eigenvalue of $k(N)$ acting on $h$;
\vskip0.1truecm\noindent
ii) $p(N)$ is the spectral projection of $N$ corresponding to the Borel set
\[
B:=\{n'\in\N: g(m)-g(n)=g(n'+\ell)-g(n')\}\subseteq sp(N)
\]
and $p(N-\ell\cdot I)$ is the spectral projection of $N$ corresponding to $B+\ell\subseteq\N$;
\vskip0.1truecm\noindent
iii) the densely defined, closed operator $(X,D(X))$ on $h$, given by
\begin{equation}
D(X):=D(N^{\ell/2})\qquad X:=p(N)\circ A^\ell,
\end{equation}
where $A$ is the annihilation operator defined in Section 5, satisfies the relations
\begin{equation}
\begin{split}
XX^*&=(N+1)\cdots (N+\ell\cdot I)p(N)\\
X^*X&=N(N-I)\cdots (N-(\ell -1)\cdot I)p(N-\ell\cdot I)\\
[X,X^* ]=&(N+1)\cdots (N+\ell\cdot I)p(N)- \\
&N(N-I)\cdots (N-(\ell -1)\cdot I)p(N-\ell\cdot I);
\end{split}
\end{equation}
iv) if $B$ is unbounded, $(X^*X,D(N^\ell))$ and $(XX^*,D(N^\ell))$ are unbounded with discrete spectra;
\vskip0.1truecm\noindent
v) if $B$ and $B+\ell$ differ by a finite set, then $([X,X^*],D(N^\ell))$ is infinitesimally small with respect to $(N^\ell,D(N^\ell))$;
\vskip0.1truecm\noindent
vi) $\xi:=X\xi_0\in L^2(h)$ is an eigenvector of $\ln\Delta_0$ with eigenvalue $\nu$:
\[
(\ln\Delta_0)\xi=\nu\cdot\xi.
\]
\end{prop}
\begin{proof}
i) follows from $sp(k(N))=\overline{\{\nu_{n'+m-n,n'}:n'\in\N\}}\subset\overline{\{\nu_{m,n}:m,n\in\N\}}\\= sp(\ln\Delta_0)$; 
ii) follows from i), the assumption on $\varepsilon$ and the Spectral Theorem; iii) by the CCR we have
\begin{equation}
NA=A(N-I),\qquad A^* N=(N-I)A^*
\end{equation}
as identities among closed operators on their common domain $D(N^{3/2})$. By induction
\[
(A^*)^\ell A^\ell=N(N-I)\cdots (N-(\ell-1)\cdot I), \qquad A^\ell (A^*)^\ell=(N+I)\cdots (N+\ell\cdot I)
\]
on the domain $D(N^\ell)$ so that, by (6.2), one gets the first relation (6.3)
\[
XX^*=p(N)A^\ell (A^*)^\ell p(N)=(N+I)\cdots (N+\ell\cdot I)p(N).
\]
Since, by (6.4), $p(N)A=Ap(N-I)$, by induction one obtains the second relation (6.3) $X^*X=(A^*)^\ell p(N)A^\ell=(A^*)^\ell A^\ell p(N-\ell\cdot I)=N(N-I)\cdots (N-(\ell-1)\cdot I)p(N-\ell\cdot I)$;
the last relation (6.3) follows by difference; iv) follows from (6.3) and the fact that $N(N-I)\cdots (N-(\ell-1)\cdot I)$ and $(N+I)\cdots (N+\ell\cdot I)$ are polynomials; v) in this case $p(N)-p(N-\ell\cdot I)$ has finite rank and $(N+1)\cdots (N+\ell\cdot I)-N(N-I)\cdots (N-(\ell -1)\cdot I)$ is polynomial of degree at most $\ell-1$; vi) since $sp(\ln\Delta_0)\cap [\nu-\varepsilon,\nu+\varepsilon]=\{\nu\}$ and ${\hat f}(\nu)=1$, by the Spectral Theorem, the spectral projection $P$ of $\ln\Delta_0$, corresponding to $\{\nu\}$, can be represented as
\[
P={\hat f}(\ln\Delta_0)=\int_\R dt f(t)e^{it\ln\Delta_0}=\int_\R dt f(t)\Delta_0^{it};
\]
%Let $\lambda_{m,n}^2:=e^{\beta(g(m)-g(n))/2}\in sp(\Delta_0^{1/2})$ for some $m\neq n$ and let $f_{m,n}:\mathbb{C}\to\mathbb{C}$ be an entire function of exponential type $|f_{m,n}(z)|\le C\cdot e^{\varepsilon |z-\ln\lambda_{m,n}^4|}$, for some $C\ge 0$ and $\varepsilon >0$,  which is square-integrable over horizontal lines.
%By the Paley-Wiener Theorem, its holomorphic Fourier transform $\hat f$ is smooth and supported in $(\ln\lambda_{m,n}^4-\varepsilon,\ln\lambda_{m,n}^4-\varepsilon)$. Suppose $\varepsilon>0$ small enough so that
%$(\ln\lambda_{m,n}^4-\varepsilon,\ln\lambda_{m,n}^4-\varepsilon)\cap sp(\ln\Delta_0)=\{\ln\lambda_{m,n}^4\}$.  The spectral projection of $\Delta_0^{1/2}$ corresponding to its eigenvalue $\lambda_{m,n}^2$ coincides with
%the spectral projection of $\ln\Delta_0$ corresponding to its eigenvalue $\ln\lambda_{m,n}^4$ and is given by
%\[
%\mathbb{E}^{\Delta_0^{1/2}}(\{\lambda_{m,n}^2\})=\mathbb{E}^{\ln\Delta_0}(\{\ln\lambda_{m,n}^4\})={\hat f}_{m,n}(\ln\Delta_0)=\int_\R dt f_{m,n}(t)e^{it\ln\Delta_0}=\int_\R dt f_{m,n}(t)\Delta_0^{it}.
%\]
since $\xi_0\in D(A^\ell)=D(N^{\ell/2})$ by (6.1), by (6.4) we have
\[
\begin{split}
\Delta_0^{it}(A^\ell(\xi_0))&=\rho_\beta^{it}\circ(A^\ell(\xi_0))\circ\rho_\beta^{-it}=\rho_\beta^{it} A^\ell \rho_\beta^{1/2}\rho_\beta^{-it}\\
&=e^{-it\beta g(N)}e^{it\beta g(N+\ell\cdot I)}A^\ell\rho_\beta^{1/2}=e^{itk(N)}(A^\ell (\xi_0)).
\end{split}
\]
Hence, $P(A^\ell\xi_0)=\int_\R dt f(t) e^{itk(N)}(A^\ell (\xi_0))=({\hat f}(k(N))A^\ell)(\xi_0)=X(\xi_0)=:\xi$ does not vanish and it is an eigenvector of $\ln\Delta_0$ corresponding to the eigenvalue $\nu$.
\end{proof}

\begin{ex}
1) If $g(t)=t$ for any $t\in\R$, $B=\N$, $p(N)=I$, $X=A^\ell$ and we reproduce the "unperturbed" case treated in Theorem 5.2.\\
2) If $g(t):=t+[t/2]$ for $t\ge 0$, $n\in 2\N$ is even and $m\in 1+2\N$ is odd, then $\ell\in 1+2\N$ is odd, $g(m)-g(n)=3\ell/2-1/2$ and $B=\{n'\in\N:g(m)-g(n)=g(n'+\ell)-g(n')\}=2\N$.\\
%1) $g(t):=t+\frac{\ell}{2\pi}\cdot\cos(2t\pi/\ell)$, $g'(t)=1-\sin(2t\pi/\ell)$, $g(n'+\ell)-g(n')=\ell$ for all $n'\in\N$, $B=\N$, $p(N)=I$, $X=A^\ell$.\\
%2)  with $|\o'(t)|\le a$ so that $g(n'+\ell)-g(n')=g(m)-g(n)$ iff $\cos(2\pi\o(n'+\ell))=\cos(2\pi\o(n'))$, i.e.
%\[
%a\ell\ge\o(n'+\ell)-\o(n')\in\mathbb{Z}.
%\]
%The spectral projection $P:={\hat f}(\ln\Delta_0)$ of the Araki Hamiltonian, corresponds to the eigenvalue $\nu$ whose multiplicity is $\natural\{(m',n')\in\N\otimes\N: g(m)-g(n)=g(m')-g(n')\}$ for any $m,n\in\N$.\\
%3) if $k(N)$ is unbounded then $B$ is finite and $p(N)$ has finite rank and so do $X^*X$ and $XX^*$.
%If $g$ is linear or sublinear, as for example $g(N)=N$ or $g(N)=\ln (N+2)$, then $XX^{\ast}$ and $X^{\ast}X$ are asymptotically linear in $N$,
%$[X,X^*]$ is asymptotically a multiple of the identity operator, the Dirichlet form $\E_X^\lambda$ has discrete spectrum. If $g$ growths faster than linearly, then $XX^{\ast}$ and $X^{\ast}X$ are asymptotically vanishing with $N$ and the Dirichlet form $\E_X^\lambda$ is bounded.
\end{ex}

%\begin{ex}
%Consider the following family of functions, parametrized by $b\in\R$, $r\ge 1$
%\[
%f(t) := \frac{e^{ibt}}{\left(\ln\left(e-1+\cosh(8\pi t)\right)\right)^r}\qquad t\in \mathbb{C}\setminus\{ik/4:k\in\mathbb{Z}\}.
%\]
%If $r>1$ then $f$ is integrable and still satisfies the periodicity relation $f(t+i/4)= e^{-b/4} f(t)$ so that conclusions similar to those of the above example can be derived. If $r=1$ then $f$ is no more integrable but in $L^p(\R)$ for any $p\in (1,2]$ and, by the Hausdorff-Young inequality, $\hat f\in L^q(\R)$ for any $q\ge 2$. If $g$ is sub-linear, one still
%has still have an eigenvector $X\xi_0$ of the modular operator operator corresponding to $\lambda:=e^{-b/4}$, interpreting (6.3) in the weak sense.
%\end{ex}

\begin{rem}
The canonical commutation relations CCR arise in the spectral analysis of
the quantum harmonic oscillator, which can be considered the canonical quantization of the classical harmonic oscillator whose phase space is the plane $\R^2$. D. Shale and W. F. Stinespring constructed in \cite{ss} a quantum system which can be regarded as the quantization of a harmonic oscillator whose phase space is the hyperbolic plane $\mathbb{H}^2$ with a fixed
negative constant curvature $k<0$. It can be also considered as a quantum
harmonic oscillator with self-interaction, the coupling constant being proportional to the curvature. In their work the authors found that the dynamics is generated by an Hamiltonian $H=\hbar\o N$ proportional to the Number Operator and that annihilation and creation operators are replaced
by operators $X$ and $X^*$ satisfying a deformed CCR
\[
[X,X^*]=\hbar\cdot I-k\hbar^2\cdot N.
\]
A similar commutation relation is satisfied by $X:=A^2$ where $A$ is the annihilation operator
\[
\begin{split}
[X,X^*]&=[A^2,(A^*)^2]=A[A,(A^*)^2]+[A,(A^*)^2]A\\
&=A[A,A^*]A^*+AA^*[A,A^*]+[A,A^*]A^* A+A^*[A,A^*]A\\
&=AA^*+AA^*+A^* A+A^*A=2I+4N.
\end{split}
\]
In reference to Section 5, $e^{-tH_{X_2}^{\lambda_2}}$, compared with the quantum Ornstein-Uhlenbeck  semigroup $e^{-tH_{X_1}^{\lambda_1}}$ (see \cite{cfl}), could be called {\it quantum Ornstein-Uhlenbeck hyperbolic semigroup}.
\end{rem}

\section{Appendix}
\subsection{Generators of a class of positivity preserving semigroups}
Let $(A,D(A))$ be a lower bounded, self-adjoint operator affiliated to a von Neumann algebra $M$ and consider the $C_0$-continuous, self-adjoint, {\it positivity preserving} semigroup on $L^2(M)$, defined by
\[
T^A_t:=e^{-tA}j(e^{-tA})=e^{-tA}Je^{-tA}J\qquad t\ge 0.
\]
If $(q_A,D(q_A))$ is the lower bounded, closed quadratic form of $(A,D(A))$, then the lower bounded, closed quadratic form of $(j(A)),JD(A))$ is given by $JD(q_A)\ni\eta\mapsto q_A[J\eta]$ and the quadratic form $(t_A,D(q_A)\cap JD(q_A))$ given by $t_A[\eta]:=q_A[\eta]+q_A[J\eta]$ is lower bounded and closed as a sum of forms sharing these same properties.
\begin{lem}
The lower bounded, closed, quadratic form of the $C_0$-continuous, self-adjoint semigroup $\{T^A_t:t\ge 0\}$ is given by $(t_A,D(q_A)\cap JD(q_A))$ and the associated self-adjoint generator, 
i.e. the generalized sum $A\dot{+}j(A)$ (see \cite{Kato}), is given by the closure $\overline{A+j(A)}$
\[
T^A_t=e^{-tA}j(e^{-tA})=e^{-t(A\dot{+}j(A))}=e^{-t\overline{A+j(A)}}\qquad t\ge 0.
\]
\end{lem}
\begin{proof}
Since $(q_A,D(q_A))$ is lower bounded, for $\eta\in L^2(M)$ the limit
\[
\lim_{t\to 0^+}t^{-1}[(\eta|(I-T^A_t)\eta)]=\lim_{t\to 0^+}t^{-1}[(\eta|(I-e^{-tA}\eta))+(e^{-tA}\eta|J(I-e^{-tA}J\eta))]
\]
exists in $\R$ if and only if both limits on the right-hand side exist in $\R$, i.e. if and only if $\eta\in D(q_A)\cap JD(q_A)$ and in this case $\lim_{t\to 0^+}t^{-1}[(\eta|(I-T^A_t)\eta)]=q_A[\eta]+q_A[J\eta]=:t_A[\eta]$. Hence the lower bounded, closed quadratic form of $\{T^A_t:t\ge 0\}$ is $(t_A,D(q_A)\cap J D(q_A))$ and this form is densely defined. As $(A,D(A))$ and $(j(A),JD(A))$ are affiliated to commuting von Neumann algebras, they strongly commute and the sum $(A+J(A),D(A)\cap J D(A))$ is densely defined, lower bounded, symmetric and essentially self-adjoint so that $A\dot{+}j(A)=\overline{A+j(A)}$.
%we may consider its Friedrichs extension $A\dot{+}j(A)$. By definition this is the self-adjoint operator associated to the closure of the lower bounded, closable form $D(A)\cap JD(A)\ni\eta\mapsto (\eta|A\eta+J(A)\eta)=q_A[\eta]+q_A[J\eta]=t_A[\eta]$. Since for any $\eta\in D(q_A)\cap JD(q_A)$, $T^A_t\eta=e^{-tA}(j(e^{-tA})\eta)=j(e^{-tA})(e^{-tA}\eta)\in D(A)\cap JD(A)$ and
%$\lim_{t\to 0^+}t_A[\eta-T^A_t\eta]=0$, it follows that $D(A)\cap JD(A)$ is a form core for $(t_A,D(q_A)\cap JD(q_A))$ so that this is the closure of the quadratic form of the sum $(A+j(A),D(A)\cap JD(A))$.
\end{proof}

\subsection{Superbounded semigroups on abelian atomic von Neumann algebras.}
The von Neumann algebra $B(h)$ is atomic and this suggests to have a look
at the superboundedness property in the abelian situation of atomic measured spaces.
\vskip0.2truecm\noindent
Let $(X,m)$ be a locally compact, second countable, Hausdorff space, endowed with a fully supported Borel measure. Consider a real valued function $U$ such that $e^{-U}\in L^1 (X,m)$ and define a probability measure by
\[
m_U:=e^{-U}m\Big/\int_X e^{- U}dm.
\]
By the unit norm function $u_0:=e^{- U/2}/\|e^{-U/2}\|_{L^2(X,m)}\in L^2(X,m)$, one recovers the integral with respect to $m_U$ by
\[
\int_Xvdm_U=(u_0|vu_0)_{L^2(X,m)},
\]
one has the embedding $i_0:L^\infty (X,m)\to L^2(X,m)$ $i_0(v):=vu_0$ with $\|i_0(w)\|_{L^2(X,m)}=\|w\|_{L^2(X,m_U)}$.
\vskip0.2truecm\noindent
A $C_0$-continuous semigroup $T_t:L^2(X,m)\to L^2(X,m)$ is Markovian with respect to $m_U$ (in the sense we are discussing in this work, i.e. the one introduced in \cite{c1}), if
\[
0\le v\le u_0\quad\Rightarrow\quad 0\le T_tv\le u_0\qquad t\ge 0.
\]
Such a semigroup induces a semigroup on the abelian von Neumann algebra $L^\infty(X,m)$ by
\[
S_t:L^\infty(X,m)\to L^\infty(X,m)\qquad i_0(S_t u)=T_t(i_0(u))\qquad u\in L^\infty(X,m),
\]
which is Markovian in the usual sense
\[
0\le u\le 1\quad\Rightarrow\quad 0\le S_tu\le 1\qquad t\ge 0.
\]
The definition of superboundedness considered above on von Neumann algebras, in the commutative setting reduces to say that $T_t$ is {\it superbounded with respect to} $m_U$ if
\[
T_t(L^2(X,m))\subset i_0(L^\infty(X,m))\qquad t>t_0
\]
for some $t_0\ge 0$ and
\[
\|u\|_{L^\infty(X,m)}\le \|v\|_{L^2(X,m)}
\]
whenever $T_t v=i_0(u)$ for $v\in L^2(X,m)$, $u\in L^\infty(X,m)$ and $t>t_0$. In other words, $T_t$ is superbounded with respect to $m_U$, if the induced Markovian semigroup satisfies
\[
\|S_t u\|_{L^\infty(X,m)}\le \|i_0(u)\|_{L^2(X,m)}=\|u\|_{L^2(X,m_U)}\qquad u\in L^\infty(X,m),\qquad t>t_0.
\]
In case $(X,m)$ is an atomic measured space, the classical definition of super or ultracontractivity typically trivializes (see \cite{d} Section 2.1): this happens, for example, if $m$ is the counting measure because of the contractive embedding $L^2(X,m)\subseteq L^\infty(X,m)$. Superboundedness
however may still be non trivial.
\vskip0.2truecm\noindent
Let $(X,m)$ be a countable, atomic measured space and let $m=e^{-h}m_0$
for some function $h$ and the counting measure $m_0$.  To simplify notations, we assume that $\|e^{-U}\|_{L^1(X,m)}=1$.
\\
For a fixed nonnegative measurable function $V:X\to [0,+\infty)$ let us consider the semigroup
\[
T_t:L^2(X,m)\to L^2(X,m)\qquad T_t v:=e^{-tV}v\qquad t\ge 0
\]
which is clearly Markovian with respect to the probability measure $m_U$.
\begin{lem}
The semigroup $T_t$ is supercontractive with respect to $m_U$ if and only
if
\[
(U+h)_+/V\in L^\infty(X,m).
\]
More precisely, $T_t$ extends to a contraction from $L^2(X,m_U)$ to $L^2(X,m)$ if and only if
\[
t\ge t_0:=\frac{1}{2}\|(U+h)_+/V\|_\infty.
\]
In case $m$ is the counting measure we have $t_0=\|U/V\|_\infty/2$.
\end{lem}
\begin{proof}
On one hand, if $t_0$ is finite and $t\ge t_0$, the result follows from $\|S_t v\|_\infty = \|ve^{-tV}\|_\infty\le \|ve^{-tV}\|_{L^2(X,m_0)}= \|ve^{(U+h)/2-tV}\|_{L^2(X,m_U)}\le \|v\|_{L^2(X,m_U)}$.
On the other hand, if $\|S_tv\|_\infty\le \|v\|_{L^2(X,m_U)}$ for some $t_0\ge 0$ and all $t\ge t_0$, choosing $v:=1_{\{x\}}$ for any $x\in X$, we have $t_0\ge \|(U+h)_+/V\|_\infty/2$.
\end{proof}

%Consider Quantum Harmonic Oscillator setup
%$([A,A^\ast]=1, N=A^\ast A)$. Let
%\[
%W\equiv A+ A^m
%\]
%for some $m>1$.
%We have
%\[
%[N,A^n]= - n A^n.
%\]
%Hence
%\[
%\alpha_{t}(W)\equiv e^{-itN}We^{itN}=e^{it}A + e^{imt}A^m
%\]
%and so
%\[
%X\equiv \int \alpha_t(W) f(t) dt= \hat{f}(1)A + \hat{f}(m)A^m
%\]
%where $\hat{f}(q)\equiv \int e^{iqt}f(t) dt$.
%Thus
%\[\begin{split}
%    \alpha_{i/4}(X)&=e^{\frac14N}(\hat{f}(1)A + \hat{f}(m)A^m)e^{-\frac14N}\\
%    &= e^{-\frac14}\hat{f}(1)A + e^{-\frac{m}4}\hat{f}(m)A^m {\color{red}?=?} \lambda X
%\end{split}
%\]
%for some $\lambda\in\mathbb{C}$.
\subsection*{Acknowledgement} The authors warmly thank the anonymous referees for their careful and scrupulous reading and suggestions.
%\newpage
\normalsize
%\begin{center} \bf REFERENCES\end{center}

\normalsize
%\begin{enumerate}

%\end{enumerate}
\end{document}